\documentclass[12pt,a4paper]{article}
\usepackage{amsmath,amssymb,amsthm,hyperref}
\usepackage{graphicx}
\usepackage{tikz}
\usepackage{xcolor}
\usepackage{version} 
\usepackage{color}
\usepackage{enumerate}
\usepackage{pstricks-add}
\usepackage{fancyhdr}
\usepackage{datetime}
\numberwithin{equation}{section}
\usepackage[bf]{caption}     
\usepackage{geometry}        
\usepackage{bm}  
\usepackage[utf8]{inputenc}
  \newtheorem{theorem}{Theorem}[section]
 \newtheorem{corollary}[theorem]{Corollary}
 
 \newtheorem{lemma}[theorem]{Lemma}
 \newtheorem{proposition}[theorem]{Proposition}
 
 {\theoremstyle{definition}
 \newtheorem{definition}[theorem]{Definition}
 \newtheorem{remark}[theorem]{Remark}

 }
 \newcommand{\tr}{}

 \definecolor{green}{rgb}{0,0.55,0}
\definecolor{lightgreen}{rgb}{0.5,0.85,-1}

\DeclareMathOperator{\lip}{lip\kern-0.8pt}
\DeclareMathOperator{\Lip}{Lip\kern-0.8pt}

\definecolor{skyblue}{rgb}{0,0.4,0.6}

\title{Big and little Lipschitz one sets}
\author{Zolt\'an Buczolich\thanks{\scriptsize
This author was supported by the Hungarian National Research, Development and Innovation Office--NKFIH, Grant  124003.
},
Department of Analysis, ELTE E\"otv\"os Lor\'and\\
University, P\'azm\'any P\'eter S\'et\'any 1/c, 1117 Budapest, Hungary\\
email: zoltan.buczolich@ttk.elte.hu\\
{\tt http://buczo.web.elte.hu/}\\
ORCID Id: 0000-0001-5481-8797
   \smallskip\\
  Bruce Hanson, Department of Mathematics,\\ Statistics and Computer Science,\\ St.\ Olaf College,
Northfield, Minnesota 55057, USA\\
{email:} \texttt{hansonb@stolaf.edu}
  \smallskip\\
 Bal\'azs Maga\thanks{\scriptsize This author was supported by the \'UNKP-18-2 New National Excellence of the Hungarian Ministry of Human Capacities, and by the Hungarian National Research, Development and Innovation Office–NKFIH, Grant 124749.},
Department of Analysis, ELTE E\"otv\"os Lor\'and\\
University, P\'azm\'any P\'eter S\'et\'any 1/c, 1117 Budapest, Hungary\\
 email: magab@cs.elte.hu \\{\tt www.cs.elte.hu/\hbox{$\sim$}magab}
   \smallskip\\
and
   \smallskip\\
 G\'asp\'ar V\'ertesy\thanks{\scriptsize This author was supported supported by the \'UNKP-18-3 New National Excellence Program of the Ministry of Human Capacities, and  by the Hungarian National Research, Development and Innovation Office–NKFIH, Grant 124749.
},
 Department of Analysis, ELTE E\"otv\"os Lor\'and\\
University, P\'azm\'any P\'eter S\'et\'any 1/c, 1117 Budapest, Hungary\\
email: vertesy.gaspar@gmail.com\
}
\date{\today}
 \begin{document}
\maketitle

\medskip

\medskip



\newpage
 \begin{abstract}{Given a continuous function 
$f: {\ensuremath {\mathbb R}}\to {\ensuremath {\mathbb R}}$ we denote the so-called ``big Lip'' and ``little lip'' functions by $ \Lip f$ and $ \lip f$ respectively}. In this paper we are interested in the following question. Given a set $E {\subset}  {\ensuremath {\mathbb R}}$ is it possible to find a continuous function $f$ such that
$ \lip f=\mathbf{1}_E$ or $ \Lip f=\mathbf{1}_E$?

For monotone continuous functions we provide the rather straightforward answer.

 For arbitrary continuous functions the answer is much more difficult   to find.   We introduce the concept of uniform density type (UDT) and
show that if $E$ is $G_\delta$ and UDT then there exists a continuous function $f$ satisfying $ \Lip f =\mathbf{1}_E$, that is, $E$ is a $ \Lip 1$ set.

{In the other direction we show that every $\Lip 1$ set is $G_\delta$ and weakly dense.  We also show that the converse of this statement is not true, namely} that there exist weakly dense $G_{{\delta}}$ sets which are not 
$ \Lip 1$.

 We say that a set $E\subset \mathbb{R}$ is $\lip 1$ if there is a continuous function $f$ such that $\lip f=\mathbf{1}_E$.  We introduce the concept of strongly one-sided density and show that every $\lip 1$ set is a strongly one-sided dense $F_\sigma$ set.  
   \end{abstract}


\section{Introduction}\label{*secintro}
Throughout this note we assume that $f: {\mathbb R} \to  {\mathbb R}$ is continuous.
Then the so-called ``big Lip'' and ``little lip'' functions are defined as follows:

 \begin{equation}
  \Lip
  f(x)= \limsup_{r\to 0^+}M_f(x,r),\qquad\label{Lipdef}
  \lip
  f(x)= \liminf_{r\rightarrow 0^+}M_f(x,r),
 \end{equation}
where
$$M_f(x,r)=\frac{\sup\{|f(x)-f(y)| \colon |x-y| \le r\}}r.$$

 As far as we know 
the definition  of $\lip f$  first appeared in  \cite{[Cheeger]} and later reappeared in \cite{[Keith]}. 

In order to connect these functions to more customary ones, after denoting the Dini derivatives by $D^{+}f(x), D^{-}f(x), D_{+}f(x), D_{-}f(x)$
(for the definitions see for example \cite{BBT} p. 317), one can easily check that
 $$\Lip f(x) = \max\{D^{+}f(x), D^{-}f(x), -D_+f(x),-D_{-}f(x) \}.$$
Note also that if $f$ is differentiable at $x$, then $\lip f(x)=\Lip f(x)=|f'(x)|$.  Moreover,  $\Lip f(x) = 0$ if and only if $f'(x)=0$.
 The connection of $\lip f(x)$ to the Dini derivatives is quite weak as the following
example shows. Suppose that $g({-1}/{2^{n^2}})=g(1/2^{n^2})=(-1)^{n}/2^{n^2}$, $n=1,2,...$
and $g(x)=0$ otherwise. The function $g$ is not continuous, but 
one can easily see that by a slight modification of $g$ one can obtain a continuous function $f$ for which
$D^{+}f(0)=D^{-}f(0)=1$, $D_{+}f(0)=D_{-}f(0)=-1$
while $\lip f(0)=0.$ 

We also define
$$L_f=\{x\in  {\mathbb R} \,:\,  \Lip f(x)<\infty \} \text{    and    }l_f=\{x\in  {\mathbb R} \,:\,  \lip f(x) < \infty \}.$$
The behaviour of the two functions, $ \Lip f$ and $ \lip f$, is intimately related to the differentiability of $f$.  For example, the Rademacher-Stepanov Theorem  \cite{MaZa} tells us that
if $ {\mathbb R}\backslash L_f$ has measure zero, then $f$ is differentiable almost everywhere on $ {\mathbb R}$.  On the other hand, in (\cite{BaloghCsornyei}, 2006) Balogh and Cs\"{o}rnyei   construct a continuous function $f: {\mathbb R}\to {\mathbb R}$ such that $ \lip f=0$ almost everywhere, but $f$ is nowhere differentiable.  However, in the same paper, they also show that if $ {\mathbb R}\backslash l_f$ is countable and $ \lip f$ is locally integrable, then $f$ is again differentiable almost everywhere on $ {\mathbb R}$.

More recently, progress has been made on characterizing the sets $L_f$ and $l_f$ for continuous functions (\cite{BHRZ}, 2018) and characterizing the sets of non-differentiability for continuous functions with either $L_f= {\mathbb R}$ or $l_f= {\mathbb R}$ (\cite{Hanson2}, 2016).  There are still a number of open problems concerning the relationship between $L_f$ ($l_f$) and the differentiability properties of $f$.

We also mention the very recent result (\cite{MaZi}, 2019) about little Lipschitz maps of analytic metric spaces with sufficiently high packing dimension onto cubes in $ {\mathbb R}^n$.

It is an interesting problem to characterize the functions $ \Lip f$ and $ \lip f$ for continuous functions $f$.  {This is in the spirit of the well-known problem of characterizing the functions $f$ which are derivatives.    (See \cite{Z}, \cite{PT}, \cite{W}, \cite{BL}.)} In this note, we take a first step in this direction by investigating when it is possible for $ \Lip f$ (or $ \lip f$) to be a characteristic function.  Given a set $E \subset  {\mathbb R}$ we say that $E$ is  $ \Lip 1$ ($ \lip 1$) if there is a continuous function $f$ defined on $ {\mathbb R}$ such that $ \Lip f =\mathbf{1}_E$, ($ \lip f=\mathbf{1}_E$).  So we are interested in determining which sets $E$ are $ \Lip 1$ or $ \lip 1$.  {(See \cite{M} for a related problem.)}

It turns out that it is straightforward to decide this in the special case where $f$ is monotone.   We say that $E$ is {\em monotone} $ \Lip 1$ ($ \lip 1$) if there is a continuous, monotone function $f$ such that $ \Lip f=\mathbf{1}_E$ ($ \lip f=\mathbf{1}_E$).  In Theorems \ref{Lip1monotonethm} and \ref{lip1monotonethm} we show that monotone $ \Lip 1$ and $ \lip 1$ sets can be characterized using simple density conditions.
The details for this are laid out in Section \ref{*seceasy}.

In Section \ref{*secLip} we see that $ \Lip 1$ sets are weakly dense $G_{{\delta}}$
sets (\tr{Definition \ref{weak dense}, } Theorem \ref{Lip1weaklydense}) and $\lip 1$ sets are strongly one-sided dense $F_{\sigma}$ sets (\tr{Definition \ref{onesideddense},} Theorem \ref{*lipnec}). In Theorem \ref{Lip1thm} we show that a set $E$ is $\Lip 1$ if and only if $\mathbb{R}$ can be divided into three sets such that they give a ternary decomposition with respect to $E$ in the sense of Definition \ref{ternary}. 
In Theorem \ref{lip 1 thm} it is proved that countable disjoint unions of closed and strongly one-sided dense sets are $ \lip 1$.

In Section \ref{*secsubsuper} we consider the more difficult problem of characterizing general $ \Lip 1$ sets. 
Given a measurable set, we introduce a two-parameter family of sets describing its levels of density and use this to define uniform density type (UDT) sets.
 \begin{definition}\label{*defudt}
Suppose that $E\subseteq\mathbb{R}$ is measurable and $\gamma,\delta>0$.
Let
 \begin{displaymath} E^{\gamma,\delta}=\left\{x\in\mathbb{R}:
 \forall r\in (0,\delta],\texttt{  }
 \max\left\{\frac{|(x-r,x)\cap E|}{r},\frac{|(x,x+r)\cap E|}{r}\right\}\geq\gamma \right\},
\end{displaymath}
where $|E|$ denotes the Lebesgue measure of the set $E$.

We say that $E$ has uniform density type (UDT) if there exist sequences $\gamma_n\nearrow 1$ and $\delta_n\searrow 0$ such that $E\subseteq \bigcap_{k=1}^{\infty} \bigcup_{n=k}^{\infty}E^{\gamma_n,\delta_n}$.
\end{definition}

Our main result from Section \ref{*secsubsuper}, Theorem \ref{*thUDTLip1}, states that $G_\delta$ sets which are UDT are $ \Lip 1$. {As we show in \cite{[BHMVlipap]}, the converse of this statement does not hold.  There exist $\Lip 1$ sets which are not UDT.}

Finally, in Section \ref{*secnotlip14} we show that the UDT condition in Theorem \ref{*thUDTLip1} cannot be replaced with one of the weaker density conditions from
Section \ref{*seceasy}.

{Summarizing the main results of this paper, we show that 
$$G_\delta+\mbox{UDT}\Rightarrow \Lip 1 \Rightarrow G_\delta+\mbox{weakly dense},$$
and that the second implication cannot be reversed.}

\section{Preliminary definitions and  results}\label{*secprel}

The union of disjoint sets $A$ and $B$ is denoted by $A\sqcup  B$.  For any $S,T\subset  {\mathbb R}$ and $x \in  {\mathbb R}$ we define $d(S,T)$ to be the lower distance from $S$ to $T$,
that is $\inf\{|x-y|:x\in S,\  y\in T \}$. Let
$d(x,S)=d(\{x\},S)$.{ (We recall that we defined $|S|$ to be the Lebesgue measure of $S$.) }

In the space of continuous functions defined on an interval $I$ we use the supremum norm $||f||=\sup\{|f(x)|: x\in I  \}$ and the metric topology generated by this norm.

 \begin{definition}\label{weak dense}
Given a sequence of non-degenerate closed intervals $\{I_n\}$, we write {\em $I_n\to x$} if $x \in I_n$ for all
$n \in  \mathbb{N}$ and $|I_n|\to 0$.

The measurable set
$E$ is {\it weakly dense} at $x$ if there exists $I_n \to x$ such that $\frac{|E \cap I_n|}{|I_n|}\to 1$.
The set $E$ is {\it weakly dense} if $E$ is weakly dense at $x$ for each $x \in E$.

The set $E$ is {\it strongly dense} at $x$ if  for every sequence $\{I_n\}$ such that $I_n \to x $ we have $ \frac{|E \cap I_n|}{|I_n|}\to 1$.
We say that  $E$ is {\it strongly dense} if $E$ is strongly dense at $x$ for each $x \in E$.

\end{definition}

Note: $E$ being strongly dense at $x$, just means that $x$ is a point of density of $E$.
\tr{(See Remark \ref{densityimplications}.)}

In this paper a.e.~always means  almost everywhere with respect to Lebesgue measure.

\tr{ \begin{lemma}\label{novekedes}
If $E\subset {\ensuremath {\mathbb R}}$ and $f\colon {\ensuremath {\mathbb R}}\rightarrow {\ensuremath {\mathbb R}}$ such that $ \lip f\le \mathbf{1}_E$ then $|f(a)-f(b)|\le |[a,b]\cap E|$ for every $a,b\in {\ensuremath {\mathbb R}}$ (where $a<b$) so $f$ is Lipschitz and hence absolutely continuous. 
\end{lemma}}

\begin{proof}
Let $ {\varepsilon}>0$. For every $x\in {\ensuremath {\mathbb R}}$ we fix $r_x\in(0, {\varepsilon})$ such that $M_f(x,r_x)<1+ {\varepsilon}$. We select a finite set $H\subset {\ensuremath {\mathbb R}}$ for which $\{(x-r_x,x+r_x) : x\in H\}$ is a minimal cover of $[a,b]$. Then every $y\in {\ensuremath {\mathbb R}}$ is contained by at most two of these open intervals. If $x\in {\ensuremath {\mathbb R}}$, $r>0$ and $y\in (x-r,x+r)$, we have $|f(x)-f(y)|\le rM_f(x,r)$. Thus
$$
|f(a)-f(b)| \le \sum_{x\in H} 2 r_x M_f(x,r_x) \le \sum_{x\in H} 2 r_x (1+ {\varepsilon}) \le 2(b+ {\varepsilon}-(a- {\varepsilon}))(1+ {\varepsilon}).
$$
Hence $f$ is Lipschitz as $a$, $b$ and $ {\varepsilon}$ were chosen arbitrarily.

Since $f$ is Lipschitz it is absolutely continuous. Therefore
 $f'$ exists almost everywhere and $f(b)-f(a)=\int_a^b  f'(t)\, dt$.
 Since $|f'| \le \mathbf{1}_E $ a.e. we obtain that
$|f(b)-f(a)|\leq \int_a^b  \mathbf{1}_E(t)\, dt=|[a,b]\cap E|$. 
 \end{proof}

\section{Necessary and/or sufficient conditions for monotone $ \Lip 1$ and $ \lip 1$ sets}\label{*seceasy}

For monotone $ \Lip 1$ and $ \lip 1$ sets it is rather easy to obtain necessary and sufficient conditions.

 \begin{theorem}\label{Lip1monotonethm}
The set $E$ is monotone $ \Lip 1$ if and only if $E$ is weakly dense and $E^c$ is strongly dense.
\end{theorem}

 \begin{proof} Assume that $E$ is monotone $ \Lip 1$.  Then we can choose a continuous, monotone increasing function $f$ such that
$  \Lip
  f=\mathbf{1}_E$.
By Lemma \ref{novekedes} $f$ is Lipschitz and therefore differentiable   a.e. 
Since $  \Lip
  f= \mathbf{1}_E
 $ and $f$ is increasing, we conclude that $f'(x)= \mathbf{1}_E
 (x)$ a.e.~and we have
 \begin{equation}\label{f E}
f(y)-f(x)=\int_x^y  \mathbf{1}_E
 (t)\, dt= |E \cap [x,y]| \text{    for all    } x < y.
\end{equation}
From (\ref{f E}) and the definition of $ \Lip f$ it is straightforward to show that
$E$ is weakly dense and $E^c$ is strongly dense.

Now assume that $E$ is weakly dense and $E^c$ is strongly dense.
Then let  $f(x)=\int_{x_0}^x  \mathbf{1}_E
 (t)\, dt$ by selecting an arbitrary $x_{0}$. It is straightforward to show that $  \Lip
  f= \mathbf{1}_E
 $ and therefore $E$ is monotone $ \Lip 1$.
\end{proof}

For the characterization of monotone  $ \lip 1$  sets we need a few new definitions:

 \begin{definition}\label{onesideddense}
Suppose that $I_n \to x$.  If each $I_n$ is centered at $x$ we say that {\it $\{I_n\}$ center converges to $x$} and we write $I_n \stackrel{c}{\to} x$.

The set $E$ is {\it weakly center dense} at $x$ if there exists a sequence $\{I_n\}$ such that
$I_n \stackrel{c}{\to} x,$ and $\frac{|E \cap I_n|}{|I_n|}\to 1$.
The set $E$ is {\it weakly center dense} if $E$ is weakly center dense at every point $x \in E$.

The set $E$ is {\it strongly one-sided dense} at $x$ if for any sequence $\{I_n\}=\{[x-r_n,x+r_n]\}$ such that
$r_n \searrow 0$  we have $\max\{\frac{|E \cap [x-r_n,x]|}{r_n},\frac{|E \cap [x,x+r_n]|}{r_n}\}\to 1$.
The set $E$ is {\it strongly one-sided dense} if $E$ is strongly one-sided dense at every point $x \in E$.
\end{definition}

 \begin{remark}\label{densityimplications}
 It is easy to see that if $E$ is right- or left-dense in the ordinary Lebesgue density sense  then it is
strongly one-sided dense. The reverse implication is not true. Indeed,    it is not difficult to see that  the set $$E=\bigcup_{n=1}^{\infty} [1/2^{(2n+1)^2},n/2^{(2n)^2}]\cup  [-n/2^{(2n+1)^2},-1/2^{(2n+2)^2}]$$
is strongly one-sided dense at $0$ according to our notation but   it is not right- or left-dense in
the ordinary Lebesgue density sense.

 The observant reader will note that we have not defined {\it strongly center dense} or {\it weakly one-sided dense}.  The reason for this is that defining
these terms in the obvious way would be redundant since strongly center dense sets would be equivalent to strongly dense sets and weakly one-sided dense sets would be equivalent to weakly dense sets.  We also observe that the following implications hold:
$$\text{    strongly dense    } \Rightarrow \text{    strongly one-sided dense,    }$$
$$\text{    weakly center dense    } \Rightarrow \text{    weakly dense    }.$$
Note that neither of the above implications is reversible: a closed interval is strongly one-sided dense and weakly dense, but not strongly dense or weakly center dense.
\end{remark}

 \begin{theorem}\label{lip1monotonethm}
The set $E$ is monotone $ \lip 1$  if and only if $E$ is strongly one-sided dense and $E^c$ is weakly center dense.
\end{theorem}

 \begin{proof}
The proof of Theorem \ref{lip1monotonethm} is {straightforward and} similar to the proof of Theorem \ref{Lip1monotonethm}. {We leave it up to the reader.}
\end{proof}

\section{Necessary and/or sufficient conditions for general $ \Lip 1$ and $ \lip 1$ sets}\label{*secLip}

In Theorem \ref{Lip1weaklydense} we give a necessary condition for a set to be $ \Lip 1$.  We will see in Section \ref{*secnotlip14} (Theorem \ref{notLip1weaklydense}) that this condition is not sufficient.

 \begin{theorem}\label{Lip1weaklydense}
If $E\subset {\ensuremath {\mathbb R}}$ is $ \Lip 1$ then $E$ is a weakly dense $G_\delta$ set.
\end{theorem}
 \begin{proof}
Suppose that $E$ is $ \Lip 1$.
Lemma \ref{novekedes} implies that $E$ is weakly dense.
Let $f\colon {\ensuremath {\mathbb R}}\rightarrow {\ensuremath {\mathbb R}}$ be such that $ \Lip f= \mathbf{1}_E$.
\tr{Since}
$$
E =  \bigcap \limits_{n=1}^\infty \Big\{x\in {\ensuremath {\mathbb R}} : \text{  there exists $r\in\Big(0,\frac{1  }{n}\Big)$ such that }M_f(x,r)>1-\frac{1}{n}  \Big\}
$$
and the sets on the right are open, we obtain that $E$ is $G_\delta$.
\end{proof}

The next definition will be used to obtain a necessary and sufficient condition for
$ \Lip 1$ sets in Theorem \ref{Lip1thm}.

 \begin{definition}\label{ternary}
Let $E$ be a measurable subset of $ \mathbb R
 $ and suppose that $E_1,E_0,E_{-1}$ are pairwise disjoint   measurable  sets whose union is $ \mathbb R
 $.
Then we say that $E_1,E_0,E_{-1}$ is a {\it ternary decomposition of $ \mathbb R
 $} with respect to $E$ if the following conditions hold:
 \begin{align}\label{E1 or E-1}
\bullet\ \ & \forall x \in E \text{    either    } E_1 \text{    or    }E_{-1} \text{    is weakly dense at    }x,\\
\bullet\ \ &\label{E1-E-1}
\forall x \notin E \text{    and    } \forall I_n \to x \text{    we have    }\frac{||E_1 \cap I_n| - |E_{-1}\cap I_n||}{|I_n|} \to 0.
\end{align}

If $E_1,E_0,E_{-1}$ is a ternary decomposition of $ \mathbb R
 $ with respect to $E$ we write $E \sim (E_1,E_0,E_{-1})$.
\end{definition}
 \bigskip

 \begin{theorem}\label{Lip1thm} A set
$E$ is $ \Lip 1$ if and only if there is a ternary decomposition of $ \mathbb R
 $ with respect to $E$.
\end{theorem}

 \begin{proof}
Suppose that $E \sim (E_1,E_0,E_{-1})$.
Define
$$f(x)=\int_0^x \mathbf{1}_{E_1}(t)-\mathbf{1}_{E_{-1}}(t) \, dt.$$
Then straightforward calculations show that $  \Lip
  f=\mathbf{1}_E$.

Working in the opposite direction, now assume that $  \Lip
  f= \mathbf{1}_E
 $.
Then by Lemma \ref{novekedes}, $f$ is  Lipschitz  and hence $f$ is differentiable almost everywhere and wherever $f'(x)$ is defined $f'(x)$ is equal to either $1,0$ or $-1$.
For $i=1,-1$ define $E_i=\{x : \, f'(x)=i\}$ and let $E_0=\mathbb R
  \backslash(E_1 \cup E_{-1})$.
  By  absolute continuity of $f$  we have that
$$f(x)=f(0)+\int_0^x  f'
 (t)\, dt=f(0)+\int_0^x \mathbf{1}_{E_1}(t)-\mathbf{1}_{E_{-1}}(t) \, dt$$ and it is straightforward to show that
$E\sim (E_1,E_0,E_{-1})$.
\end{proof}
 \bigskip

 \begin{remark}\label{*remtre} Suppose that $E \sim (E_1,E_0,E_{-1})$.
Then we can find $F_1,F_0,F_{-1}$ such that
$E \sim (F_1,F_0,F_{-1})$ and $E = F_1 \cup F_{-1}$.
\end{remark}

To verify that Remark \ref{*remtre} is true assume that $E \sim (E_1,E_0,E_{-1})$.
\tr{By the Lebesgue density theorem} almost every element in a set is a density point of the set and hence it follows from the definition of a ternary decomposition that
$|E_i \backslash E|=0$ for $i =1,-1$ and $|E_0 \cap E|=0$.
Thus, if we define
$F_1=E\backslash E_{-1}$, $F_{-1}=E_{-1} \cap E$, and $F_0=\mathbb R \setminus E$, then we have $E \sim (F_1,F_0,F_{-1})$ and $E = F_1 \cup F_{-1}$.
 \bigskip

 \begin{remark}\label{verify}
Although Theorem \ref{Lip1thm} gives a characterization of $ \Lip 1$ sets, it is not always easy to verify whether or not a given set $E$ has a ternary decomposition.  One simple example is
$E=(0,\infty)$.  In this case, one can verify that $E_0=(-\infty,0]$, $E_{-1}=\bigcup_{n=1}^\infty (\frac1{2n+1},\frac1{2n}]$,
$E_1=(\bigcup_{n=1}^\infty (\frac1{2n},\frac1{2n-1}])\cup (1,\infty)$ gives a ternary decomposition of $ {\mathbb R}$ with respect to $E$ and therefore $E$ is $ \Lip 1$.
\end{remark}


Next we want to find some necessary and some sufficient conditions for $\lip 1$ sets. 
\tr{For this purpose we will need to use} Lemma \ref{novekedes}.

 \begin{theorem}\label{*lipnec}
If $E\subset {\ensuremath {\mathbb R}}$ is $ \lip 1$ then $E$ is a strongly one-sided dense $F_\sigma$ set.
\end{theorem}

 \begin{proof}
Suppose that $E$ is $ \lip 1$. Lemma \ref{novekedes} implies that $E$ is strongly one-sided dense. 
Let $f\colon {\ensuremath {\mathbb R}}\rightarrow {\ensuremath {\mathbb R}}$ such that $ \lip f= \mathbf{1}_E$.
As 
$$
E^c =  \bigcap_{n=1}^\infty \Big\{x\in {\ensuremath {\mathbb R}} : \text{  there exists $r\in\Big(0,\frac{1  }{n}\Big)$ such that }M_f(x,r)<\frac{1}{2}  \Big\}
$$
and the sets on the right are open, the set $E^c$ is $G_\delta$ hence $E$ is $F_\sigma$. 
\end{proof}

{Theorem \ref{*lipnec} provides a necessary condition for a function to be $\lip 1$.  The following result provides a sufficient condition.}

  \begin{theorem}\label{lip 1 thm}
Suppose that $E=\sqcup_{n=1}^\infty E_n$ where for each $n \in  \mathbb{N}
 $,  $E_n$ is closed and strongly one-sided dense.
Then $E$ is $ \lip 1$.
\end{theorem}

We should note that simple examples show that the converse of  Theorem \ref{lip 1 thm} does not hold.  For example, non-empty, open sets are $ \lip 1$, but no non-empty, open set can be expressed as a disjoint, countable union of closed sets.

 We note that apart from unions of non-degenerate closed intervals, it is not completely trivial to construct closed and strongly one-sided dense sets.
 However, in Theorem 3.1 of \cite{[BHMVlipap]} we construct a nowhere dense closed
set which has SUDT  and hence is   strongly one-sided dense.

 \begin{remark}\label{*rennLo}  A result analogous to Theorem \ref{lip 1 thm}
 for $\Lip 1$ sets is not true.
 If $E$ is dense in $ \mathbb R
 $ and each $E_n$ is nowhere dense, then $E$ is not $ \Lip 1$. Indeed, according to Theorem \ref{Lip1weaklydense}, if $E$ were $\Lip 1$, it would be $G_{\delta}$ as well. However, a dense $G_{\delta}$ set cannot be written as the countable union of nowhere dense sets according to Baire's category theorem.
\end{remark}

The proof of Theorem \ref{lip 1 thm} {depends} on the following:

 \begin{lemma}\label{small lip 1}
Suppose that $E$ is closed and strongly one-sided dense.
Let $\varepsilon  >0$.  Then there exists a continuous function $f$ such that
 \begin{enumerate}[(i)]
\item\label{kar} $  \lip  f =\mathbf{1}_E$,
\item\label{f kicsi} $0 \le f(x) \le \varepsilon $ for all $x \in  \mathbb R$.
\end{enumerate}
\end{lemma}

 \begin{proof}[Proof of Lemma \ref{small lip 1}]
For every $i\in  {\ensuremath {\mathbb Z}}$ define $E_i=[(i-1) {\varepsilon},i {\varepsilon}]=[a_{i-1},a_i]$ and choose $x_i\in E_i$ such that $|E\cap [a_{i-1},x_i]|=|E\cap[x_i,a_i]|$.
For each $i\in\mathbb{Z}$ define $E^+_i=E\cap[a_{i-1},x_i]$ and $E^-_i=E\cap[x_i,a_i]$ and let $E^+=\bigcup_{i=-\infty}^\infty E^+_i$ and $E^-=\bigcup_{i=-\infty}^\infty E^-_i$.
Define $f(x)=\int_0^x \mathbf{1}_{E^+}(t)-\mathbf{1}_{E^-}(t)\, dt$.
\tr{It is easy} to verify that \eqref{kar} and \eqref{f kicsi} hold.
\end{proof}

 \begin{proof}[Proof of Theorem \ref{lip 1 thm}]
Assume that $E=\sqcup_{n=1}^\infty E_n$, where each $E_n$ is closed and strongly one-sided dense.
\tr{Redefining} $E_1,E_2,\ldots$ we can suppose that if $n\ge 2$ then $E_n$ is bounded.   

Using Lemma \ref{small lip 1} we choose $f_1$ such that
\eqref{kar} and \eqref{f kicsi} of Lemma \ref{small lip 1} hold with $f$ replaced by $f_1$, $E$ replaced by $E_1$ and $\varepsilon =1$.
Using Lemma \ref{small lip 1} for each $n \in  \mathbb{N}\cap [2,\infty)$ we recursively choose $f_n\geq 0$ such that
(\ref{kar}) holds with $f$ replaced by $f_n$ and $E$ replaced with $E_n$ such that
 \begin{equation}\label{*fnnx}
0 \le f_n(x) \le 2^{-n} \min \bigg\{1, d   \bigg(E_n,\bigcup_{k=1}^{n-1} E_k \bigg) \bigg\}
\end{equation}
(we note that the right-hand side is positive, as $E_n$ is compact and $\bigcup_{k=1}^{n-1} E_k$ is closed). Obviously, for every $n\in {\ensuremath {\mathbb N}}$
 \begin{equation}\label{komplementeren konst}
\text{  $f_n$ is constant on each interval contiguous to $E_n$.  }
\end{equation}
Let $f(x)=\sum_{n=1}^\infty f_n$.

Suppose that $x\in E_{n_0}$ for some $n_0\in {\ensuremath {\mathbb N}}$, and $ {\varepsilon} >0$.  
Using 
\begin{equation}\label{f_n p}
\text{$f_{n}(s)\geq 0$ for any $n$ and $s$ }
\end{equation}
we infer
\begin{equation}\label{*fnox}
\Big ||f(x)-f(y)|-|f_{n_{0}}(x)-f_{n_{0}}(y)| \Big |
{\le} \Big |\sum_{n\ne n_{0}}( f_{n}(x)- f_{n}(y)) \Big | \underset{\eqref{f_n p}}\leq 
\sum_{\substack{n\ne n_0 }} \sup \limits_{t\in {\ensuremath {\mathbb R}}} f_n(t).
\end{equation}
Let 
 \begin{equation}\label{*45*}
\text{  $n_1 := \max\{n_0+1,- \lfloor\log_2 ( {\varepsilon  })  \rfloor\}$
and $r\in  \bigg ( 0, d   \Big  ( \{x\},\bigcup_{n\in {\ensuremath {\mathbb N}}\cap[1,n_1]\setminus\{n_0\}} E_n  \Big  ) \bigg )$.}
\end{equation}

For every $y\in[x-r,x+r]$ we have
 \begin{equation*}
 \begin{gathered}
\Big|\dfrac{|f(x)-f(y)|}{r}-\frac{|f_{n_0}(x)-f_{n_0}(y)|}{r}\Big| 
\underset{ \eqref{komplementeren konst}  \text{ and } \eqref{*fnox}}{\le} 
 r^{-1}\sum_{\substack{n\ne n_0 \\ E_n\cap (x-r,x+r) \neq \emptyset}} \sup \limits_{t\in {\ensuremath {\mathbb R}}} f_n(t) \\
\underset{\text{   \eqref{*fnnx}}}{\le} r^{-1} \sum_{\substack{n\ne n_0 \\ E_n\cap (x-r,x+r) \neq \emptyset}}  2^{-n} d (E_n,E_{n_0}) 
\le r^{-1} \sum_{\substack{n\ne n_0 \\ E_n\cap (x-r,x+r) \neq \emptyset}}  2^{-n} d (E_n,\{x\})   \\
\le r^{-1} \sum_{\substack{n\ne n_0 \\ E_n\cap (x-r,x+r) \neq \emptyset}} 2^{-n}r
\underset{\text{   \eqref{*45*}}}{\le}
 \sum \limits_{n=n_1+1}^\infty 2^{-n}
= 2^{-n_1}
\le  {\varepsilon}.
\end{gathered}
\end{equation*}
Thus $ \lip f(x) =  \lip f_{n_0}(x) = 1$.

Now let $x\notin E$ and $\varepsilon>0$  be  arbitrary.
If $x$ is not an accumulation point of $E$ then obviously $ \lip f(x) = 0$.
Otherwise, set $n_1 := \max\{1,- \lfloor\log_2 ( {\varepsilon})  \rfloor+1\}$ and $r :=  d  \big(\{x\},\bigcup_{1 \le n \le n_1} E_n \big)$.
For every $y\in[x-r,x+r]$
 \begin{equation*}
 \begin{gathered}
\dfrac{|f(x)-f(y)|}{r}
\underset{\eqref{komplementeren konst}  \text{ and } \eqref{f_n p}}{\le} r^{-1}\sum_{\substack{n\in {\ensuremath {\mathbb N}} \\ E_n\cap (x-r,x+r) \neq \emptyset}} \sup \limits_{t\in {\ensuremath {\mathbb R}}} f_n(t) \\
\underset{\text{   \eqref{*fnnx}}}{\le} 
r^{-1} \sum_{\substack{n\in {\ensuremath {\mathbb N}} \\ E_n\cap (x-r,x+r) \neq \emptyset}} 2^{-n} d (E_n,\bigcup_{1\leq k \leq n_1} E_k)   \\
\le r^{-1} \sum_{\substack{n\in {\ensuremath {\mathbb N}} \\ E_n\cap (x-r,x+r) \neq \emptyset}} 2^{-n}\cdot 2r \le \sum \limits_{n=n_1+1}^\infty  2\cdot2^{-n} = 2\cdot2^{-n_1} \le  {\varepsilon} .
\end{gathered}
\end{equation*}
Since $r\rightarrow 0$ as $n_1\rightarrow\infty$ (and $n_1\rightarrow\infty$ as $ {\varepsilon}\rightarrow 0$), we obtain $ \lip f(x) = 0$. 
\end{proof}

 \bigskip

The refereeing procedure for this paper took a while and during this time in \cite{BHMV2} we managed to obtain a characterization of  $\lip 1$ sets as countable unions of closed sets which are strongly one-sided dense. The above special case in
Theorem \ref{lip 1 thm} has  a  simpler proof  than the main result of 
\cite{BHMV2}.

\section{$G_{{\delta}}$ uniform density type sets are $ \Lip 1$}
\label{*secsubsuper}

\tr{In this section we prove our main result: Theorem \ref{*thUDTLip1}, which asserts that $G_\delta$ sets which are also UDT are $\Lip 1$.}

Recall that the sets $E^{\gamma,\delta}$ were defined in Definition \ref{*defudt}.

 \begin{lemma} \label{closed}
 For any $\gamma,\delta>0$ the set $E^{\gamma,\delta}$ is closed.
 \end{lemma}

 \begin{proof} For $r>0$ we introduce the notation
 \begin{displaymath}
E_r^{\gamma}=\left\{x\in\mathbb{R}:\max\left\{\frac{|(x-r,x)\cap E|}{r},\frac{|(x,x+r)\cap E|}{r}\right\}\geq\gamma\right\}.
\end{displaymath}
Then we obviously have
 \begin{displaymath}
E^{\gamma,\delta}= \bigcap_{r:0<r\leq\delta}E_r^{\gamma}.
\end{displaymath}
\tr{Note that} the functions $x \mapsto \frac{|(x-r,x)\cap E|}{r}$ 
and $x \mapsto \frac{|(x,x+r)\cap E|}{r}$
are obviously continuous for any $r$ and hence
 \begin{displaymath}
x \mapsto \max\left\{\frac{|(x-r,x)\cap E|}{r},\frac{|(x,x+r)\cap E|}{r}\right\}
\end{displaymath}
is also continuous, which immediately yields that each upper level set $E_r^{\gamma}$ is closed.
Consequently, their intersection $E^{\gamma,\delta}$ is also closed. \end{proof}

\begin{proposition}\label{*propudt}
UDT sets are strongly one-sided dense.
\end{proposition}

\begin{remark}\label{*expropudt}
There are strongly one-sided dense sets which are not UDT, however the construction
of such sets is not that easy.  We write a short note, \cite{bhmvdensity} on this topic.
\end{remark}

\begin{proof}[ Proof of Proposition \ref{*propudt}]
\tr{Suppose $E$ is UDT. Then there exist sequences $\gamma_n\nearrow 1$ and $\delta_n\searrow 0$ such that $E\subseteq \bigcap_{k=1}^{\infty} \bigcup_{n=k}^{\infty}E^{\gamma_n,\delta_n}$. Let $x\in E$ and $\gamma<1$.} Choose $k$  such that $\gamma_{n}>\gamma$ when $n\geq k$. 
Then there exists $n(\gamma, x)\geq k$  such that $x\in E^{\gamma_{n(\gamma, x)},\delta_{n(\gamma, x)}}$,   that is  
$$\text{$\max\Big\{\frac{|(x-r,x)\cap E|}{r},\frac{|(x,x+r)\cap E|}{r}\Big\}>
\gamma_{n(\gamma, x)}>
\gamma$ holds for $0<r< \delta_{n(\gamma, x)}$.}$$ Since this is true  for any $0<\gamma<1$ we see that $E$ is strongly one-sided dense at $x$.
\end{proof}

\tr{The following notion is closely related to UDT.}

 \begin{definition}\label{*defsudt}
We say that $E$ has strong uniform density type (SUDT) if there exist sequences $\gamma_n\nearrow 1$ and $\delta_n\searrow 0$ such that $E\subseteq \bigcup_{k=1}^{\infty} \bigcap_{n=k}^{\infty}E^{\gamma_n,\delta_n}$.
\end{definition}

\tr{For the following proposition assume that all sets which occur in its statement are measurable subsets of $\mathbb{R}$.}

 \begin{proposition} \label{*prop06}

 \begin{enumerate}[(i)]
\item\label{*eni} If a set $E$ has SUDT then it also has UDT.
\item Any interval has SUDT (and hence UDT).
\item If $E_1, E_2, ...$ have UDT (resp.
SUDT) then $E= \bigcup_{m=1}^{\infty}E_m$ also has UDT (resp.
SUDT).
\item There exists $E$ which has SUDT but its closure $\overline{E}$
is not strongly one-sided dense and hence
 does not have UDT.
\end{enumerate}
\end{proposition}

 \begin{proof}
\tr{Statements}
(i) and (ii) are obvious.

In (iii) we will examine the UDT case, the proof of the SUDT case is basically the same.
\tr{For each set $E_m$ choose a pair of sequences $(\gamma_{m,n})_{n=1}^\infty, (\delta_{m,n})_{n=1}^\infty$ such that $\gamma_{m,n}\nearrow 1$ and $\delta_{m,n}\searrow 0$ such that $E_m\subseteq \bigcap_{k=1}^{\infty} \bigcup_{n=k}^{\infty}E^{\gamma_{m,n},\delta_{m,n}}$.}
\tr{ Then a straighforward diagonalization argument shows that we can choose} sequences $(\gamma_n)_{n=1}^\infty$ and $(\delta_n)_{n=1}^\infty$ such that $\gamma_n\nearrow 1$ and $\delta_n\searrow 0$ and for every $m\in {\ensuremath {\mathbb N}}$ there is an $n_m\in {\ensuremath {\mathbb N}}$ for which
\begin{equation}\label{*brubru}
\text{  for all $n>n_m$ we have $0<\delta_n<\delta_{m,n  }$ and $\gamma_n<\gamma_{m,n}<1$.}
\end{equation}

Thus for every $m\in {\ensuremath {\mathbb N}}$ and $n>n_m$ we have that $E^{\gamma_{m,n},\delta_{m,n}}\subseteq E^{\gamma_n,\delta_n}$, hence
 \begin{displaymath}
E_m\subseteq  \bigcap_{k=1}^{\infty} \bigcup_{n=k}^{\infty}E^{\gamma_{m,n},\delta_{m,n}}
=  \bigcap_{k=n_m}^{\infty} \bigcup_{n=k}^{\infty}E^{\gamma_{m,n},\delta_{m,n}}
\subseteq  \bigcap_{k=n_m}^{\infty} \bigcup_{n=k}^{\infty}E^{\gamma_n,\delta_n}
=  \bigcap_{k=1}^{\infty} \bigcup_{n=k}^{\infty}E^{\gamma_n,\delta_n}.
\end{displaymath}
This implies
 \begin{displaymath}
E\subseteq  \bigcap_{k=1}^{\infty} \bigcup_{n=k}^{\infty}E^{\gamma_n,\delta_n},
\end{displaymath}
that is $E$ has UDT.

Finally, for (iv) consider
 \begin{displaymath}
E= \bigcup_{n=-\infty}^{\infty}\left[2^n-2^{n-2},2^n\right].
\end{displaymath}
By (iii), ${E}$ has SUDT as it is a countable union of intervals.
\tr{Note that} its closure is $\overline{E}=E\cup\{0\}$.
But in intervals of the form $\left(0,2^n-2^{n-2}\right)$, $n\in  {\ensuremath {\mathbb Z}}$ the set $\overline{E}$ has density
 \begin{displaymath}
\frac{1}{2^n-2^{n-2}}\sum_{k=-\infty}^{n-1}2^{k-2}=\frac{2^{n-2}}{2^n-2^{n-2}}=\frac{1}{3},
\end{displaymath}
and for any interval of the form $(-r,0)$ for $r>0$ the set $\overline{E}$ has density 0. Consequently,  $E$ is not strongly-one sided dense at $0$ and therefore not UDT.  
\end{proof}

\tr{The following theorem is the main result of this paper.}

 \begin{theorem}
\label{*thUDTLip1}
 Assume that $E$ is $G_\delta$ and $E$ has UDT. Then there exists a continuous function $f$ satisfying $ \Lip f =\mathbf{1}_E$, that is, the set $E$ is $ \Lip 1$.
\end{theorem}

In order to prove the theorem we will need a pair of  definitions and a couple of technical lemmas:

\tr{ \begin{definition}\label{*defvic}
By a vicinity $U$ of a function $f\colon \mathbb{R}\rightarrow\mathbb{R}$ we mean a set of functions of the following form:
 \begin{displaymath}
U=\{g: \forall x \text{    } |f(x)-g(x)|\leq r(x)\},
\end{displaymath}
where $r(x)$ is a fixed, continuous, nonnegative function, called the radius of $U$.
\end{definition}}

 \begin{definition}\label{envelope}
Suppose that $f$ is continuous on the interval $[a,b]$ and $f_l,f_u$ are continuous on $[a,b]$ with $f_l < f < f_u$ on $(a,b)$ and $f_l(a)=f_u(a)=f(a)$ and $f_l(b)=f_u(b)=f(b)$.  Then we say that $(f_l,f_u)$ is an {\em  envelope} for $f$ on $[a,b]$ and we write $f \in (f_l,f_u)$ on $[a,b]$.
\end{definition}

For each of the following lemmas we assume that $E$ is as in the statement of Theorem \ref{*thUDTLip1} and that $\phi(x)=\int_0^x \mathbf{1}_E(t) \, dt.$
 Observe that $\phi(y)-\phi(x) = |[x,y]\cap E|$ for $x\le y$.  

 \begin{lemma}\label{envelope1}
Assume that $f$ is continuous and monotone on $[a,b]$ and $f\in (f_l,f_u)$ on $[a,b]$.  Furthermore, let $0 < \delta <  \epsilon \le 1$ and assume that
 \begin{equation}\label{1-epsilonphi}
|f(x)-f(y)|\le (1- \epsilon)|\phi(x)-\phi(y)| \text{    for all    }x,y \in [a,b].
\end{equation}
Then, there exists a continuous function $g$ on $[a,b]$ such that
\begin{itemize}
\item $g \in (f_l,f_u)$ on $[a,b]$,
\item $g$ is locally piecewise monotone on $(a,b)$, that is any compact subinterval  of $(a,b)$  can be divided into finitely many subintervals on each of which $g$ is monotone,
\item On any interval of monotonicity of $g$
there exists a constant $K$ depending only on the interval such that
 $g=K\pm (1-\delta)\phi$.
\end{itemize}
\end{lemma}

 \begin{proof}
We first note the following useful fact, which follows from the inequalities $0 <\delta <  \epsilon$ and inequality (\ref{1-epsilonphi}):

 Given any interval $[r,s] \subset (a,b)$ we can choose $t\in (r,s)$ such that
  \begin{equation}\label{*trs*}
 (1-\delta)(|E \cap [r,t]|-|E \cap [t,s]|)=f(s)-f(r).
 \end{equation}
 
 Next, we note that in order to prove the lemma it suffices to prove that for any subinterval $[c,d]\subset (a,b)$ we can construct a continuous function $g$ on $[c,d]$ such that
  \begin{enumerate}[(i)]
 \item $f_l(x) < g(x) < f_u(x)$ on $[c,d]$,
\item $g(c)=f(c)$ and $g(d)=f(d)$,
 \item $g$ is piecewise monotone on $[c,d]$,
\item $g=K\pm  (1-\delta) \phi$ on each interval of monotonicity of $g$,
recall that we use constants
$K$ which depend on the interval considered.
  \end{enumerate}

 Assume that $[c,d] \subset (a,b)$.  Let
 $$\gamma=\inf_{c\le x \le d}\min\{(f_u(x)-f(x)),(f(x)-f_{l}(x))\}>0.$$ Using the uniform continuity of $f_{u}$ and $f_{l}$ on $[c,d]$ choose a positive integer $n$ such that
 \begin{equation}\label{1-delta ineq}
\begin{gathered}\frac{d-c}n<\frac{\gamma} 3 \text{  and for  }x,y\in [c,d],\  |x-y|<\frac{d-c}{n}  \text{  we have  } \\
\max \{|f_{u}(x)-f_{u}(y)|, |f_{l}(x)-f_{l}(y)|\}<\frac{\gamma}{3}.
\end{gathered}
\end{equation}

For $i=0,1,2,\ldots, n$ let $c_{2i}=c+i(\frac{d-c}n)$ so we have $c=c_0<c_2<c_4<\ldots<c_{2n}=d$.  Using \eqref{*trs*} for each $i=1,2,\ldots, n$ we choose $c_{2i-1} \in (c_{2i-2},c_{2i})$ such that
 \begin{equation}\label{ci eq}
(1-\delta)(|E \cap [c_{2i-2},c_{2i-1}]|-|E \cap [c_{2i-1},c_{2i}]|)=f(c_{2i})-f(c_{2i-2}).
\end{equation}
Next, for each $j=0,1,2,\ldots,2n-1$ we define $g$ in $[c_j,c_{j+1}]$ by
\tr{$$
g(x)=\begin{cases}
(1-\delta)(\phi(x)-\phi(c_j))+f(c_j) \text{    if    } j \text{    is even    }\\
-(1-\delta)(\phi(c_{j+1})-\phi(x))+f(c_{j+1}) \text{    if    } j \text{    is odd.    }
\end{cases}$$}
We see that for $j=0,1,2,\ldots, n-1$ we have $g$ is monotone increasing on $[c_{2j},c_{2j+1}]$  and monotone decreasing on  $[c_{2j+1},c_{2j+2}]$.  Furthermore,
  $g=K_i\pm (1-\delta)\phi$ on each interval $[c_i,c_{i+1}]$ for $i=0,1,2,\ldots, 2n-1$ with suitable constants $K_{i}$.    We also see that (ii) holds and (i) follows from inequality \eqref{1-delta ineq}.  This concludes the proof of the lemma.
\end{proof}

 \begin{lemma}\label{envelope2}
Suppose that $f$ is continuous on $[a,b]$, $f\in (f_l,f_u)$ on $[a,b]$, and
$H$ is a closed set such that $H \subset (a,b)\backslash E$.  Furthermore, assume that $0 < \delta <  \epsilon \le 1$ and
(\ref{1-epsilonphi}) holds.
Then there exists a function $g$ continuous on $[a,b]$ with
 \begin{enumerate}[(i)]
\item $g\in (f_l,f_u)$ on $[a,b]$,
\item $g(a)=f(a)$, $g(b)=f(b)$,
\item  $g'=0$ on $H$,
\item $|g(x)-g(y)|\le (1-\delta)|\phi(x)-\phi(y)|$ for all $x,y \in [a,b]$.
 \end{enumerate}
\end{lemma}

 \begin{proof}

Write $(a,b)$ as a countable union of non-overlapping closed intervals $[c,d]$ which satisfy
 \begin{equation}\label{flcdfu}
f_l(x)<\min\{f(c),f(d)\}\le \max\{f(c),f(d)\} < f_u(x) \text{    for all    }x \in [c,d].
\end{equation}
Assume that $[c,d]$ is a closed subinterval of $(a,b)$ satisfying (\ref{flcdfu}).  
It suffices to show that we can define $g$ on $[c,d]$ such that
 \begin{equation}\label{gequalsf}
g(c)=f(c) \text{    and    }g(d)=f(d),
\end{equation}
 \begin{equation}\label{gprimezero}
g'=0 \text{    on    } H \cap [c,d],
\end{equation}
 \begin{equation}\label{iv}
\text{    (iv) holds with    }[a,b] \text{    replaced by    }[c,d],
\end{equation}
 and
  \begin{equation}\label{flgfu}
 f_l(x) < g(x) <f_u(x)\text{    for all    }x \in [c,d].
 \end{equation}  
\tr{We treat the case where $c,d \in H$.  If either $c$ or $d$ is not in $H$, one can proceed similarly: this argument is left to the reader.}
We note that
if $|E \cap (c,d)| = 0$, then it follows from inequality (\ref{1-epsilonphi})  that $f$ is constant on $[c,d]$ and we simply define $g=f$ on $[c,d]$.  Thus we may as well assume that $|E \cap (c,d)|>0$.  We also assume without loss of generality that $f(d)\ge f(c)$.  Next we choose finitely many intervals $I_i=(c_{2i-1},c_{2i})$, $i=1,2,\ldots,n$ which are contiguous to $H\cap [c,d]$ and such that $c\le c_1< c_2 \le c_3 < c_4 \le \ldots \le c_{2n-1}<c_{2n}\le d$ and
 \begin{equation}\label{1-delta1-epsilon}
(1-\delta)|E \cap(\bigcup_{i=1}^n[c_{2i-1},c_{2i}])|>(1- \epsilon)|E\cap[c,d]|\ge f(d)-f(c).
\end{equation}
Furthermore, we choose $\gamma$ such that
$$\gamma|E\cap \bigcup_{i=1}^n [c_{2i-1},c_{2i}]|=f(d)-f(c),$$
and note that $\gamma < 1-\delta$.  Using this fact, on each interval $[c_{2i-1},c_{2i}]$ one can define a monotone function $g_i$ so that
$$g_i'(c_{2i-1})=g_i'(c_{2i})=0$$
$$g_i(c_{2i-1})=0 \text{    and    }g_i(c_{2i})=\gamma |E \cap[c_{2i-1},c_{2i}]|,$$
and
$$|g_i(x)-g_i(y)|\le (1-\delta)|\phi(x)-\phi(y)| \text{    for all    }x,y \in [c_{2i-1},c_{2i}].$$
We also extend $g_i$ to the entire interval $[c,d]$ by defining $g_i=0$ on $[c,c_{2i-1}]$ and
$g_i=g_i(c_{2i})$ on $[c_{2i},d]$.  Finally, we define $g=f(c)+\sum_{i=1}^n g_i$ on $[c,d]$.  Then using (\ref{flcdfu}) it is straightforward to verify that (\ref{gequalsf})-(\ref{flgfu}) hold and we are done with the proof.
\end{proof}

 \begin{proof}[Proof of Theorem \ref{*thUDTLip1}]
\tr{Throughout this proof the value of the constant $K$ will depend on the interval with which  it is  associated.}
Fix sequences $\gamma_n$ and $\delta_n$ witnessing the UDT property of $E$.
Let $E= \bigcap_{n=1}^{\infty}G_n$, where each set $G_n$ is open and $G_{n+1}\subset G_n$ for all $n \in  {\mathbb {N}}$.  We also assume, as we may, that
each component of $G_n$ intersects $E$.
We also denote the complement of $G_n$ by $F_n$.
Thus $(F_n)_{n=1}^{\infty}$ is an increasing sequence of closed sets.
Let $\phi(x)=\int_{0}^{x}\mathbf{1}_E(t)dt$ be the integral function of the characteristic function of $E$.
We will construct a sequence of functions $(f_n)_{n=1}^{\infty}$ together with a sequence of vicinities $(U_n)_{n=1}^\infty$ with the following properties:
(Recall that the vicinity was defined in Definition \ref{*defvic}.)
 \begin{enumerate}[(i)]
\item $f_n$ is differentiable on $F_n$ and its derivative vanishes there.
\item For any $m\geq{n}$ we have $f_m\restriction _{F_n} = f_n\restriction _{F_n}$.
\item For any $x,y\in\mathbb{R}$ we have $|f_n(x)-f_n(y)|\leq(1-2^{-3n})|\phi(x)-\phi(y)|$.
\item If $(a,b)$ is an  interval contiguous to $F_n$, then for any $x\in E^{\gamma_n,\delta_n}\cap(a,b)$ there exists $y_n(x)\in(a,b)$ such that $|x-y_n(x)|\le\delta_n$ and $|f_n(x)-f_n(y_n(x))|>(1-2^{-2n})\gamma_n|x-y_n(x)|$.  Moreover, $y_n(x)$ may be chosen so that $|y_n(x)-x|$ is bounded away from 0 on each compact subset of $(a,b)$.
    Additionally, $f_n$ is locally monotone on $(a,b)$ and on each interval of monotonicity we have  $f_n=K\pm (1-2^{-3n})\phi$.
\item Each $U_n$ has a continuous radius $r_n$ satisfying $r_n(x)\le \min\{2^{-n},d(x,F_n)^2\}$ for all $x \in  {\mathbb R}$ and $r_n(x)>0$ for all $x \in G_n$ and $r_n(x),  r_n(y_n(x)) < 2^{-2n}\gamma_n|x-y_n(x)|$ for all $x \in E^{\gamma_n,\delta_n}$.  Moreover, $f_m \in U_n$ for all $m \ge n$ and $U_{n+1}\subset U_n$ for all $n \in  {\mathbb {N}}$.
\item For any $m\geq{n}$ and $x \in G_n \cap E^{\gamma_n,\delta_n}$ we have $$|f_m(x)-f_m(y_n(x))|>(1-2^{-n})\gamma_n|x-y_n(x)|.$$
\item For any $g \in U_n$ we have $g'=0$ on $F_n$.
\end{enumerate}

Assume for the moment that we have established (i)-(vii).  From (v) it follows that the $f_n$s converge uniformly to some function $f$.  Then, $f \in U_n$ for all $n \in  {\mathbb {N}}$ so by (vii) we may conclude that $f'=0$ on $\bigcup_{n=1}^\infty F_n= {\mathbb R}\backslash E$ and therefore $ \Lip f=0$ on $ {\mathbb R} \backslash E$.
On the other hand, if $x\in E$, we have $x\in E^{\gamma_n,\delta_n}$ for infinitely many choices of $n$ as $E$ has UDT.
Hence by (iv) and (vi) there exists $y_n(x)$ satisfying $|x-y_n(x)|<\delta_n$ and $|f_m(x)-f_m(y_n(x))|>(1-2^{-n})\gamma_n|x-y_n(x)|$ for all $m \ge n$, which yields $|f(x)-f(y_n(x))|\ge(1-2^{-n})\gamma_n|x-y_n(x)|$. As we have $(1-2^{-n})\gamma_n\to 1$, we deduce that $ \Lip f(x)\ge 1$.  On the other hand, by (iii), $ \Lip f \le 1$ everywhere and we have $ \Lip f(x)=1$ for $x\in E$ which concludes the proof.

In the following we construct the $f_n$s and $U_n$s and verify that (i)-(vii) are valid.  We begin by constructing $f_1$ and then define the other functions recursively.

 \begin{figure}[!htb]
 \begin{center}
\includegraphics[trim={0 0 5cm 4cm}, width=0.95 \textwidth, clip]{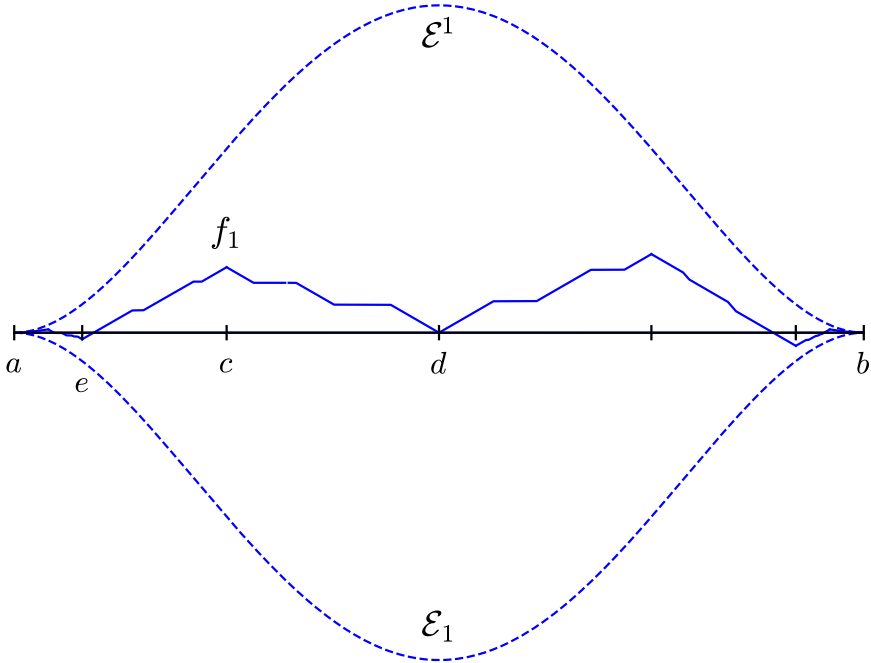}
\vspace*{-1cm}
\caption{definition of $f_{1}$ on $(a,b)$}\label{*fig1}
\end{center}
\end{figure}

To begin we set $f_0=f_0^*\equiv 0$ and we also define $f_1=0$ on $F_1$. Set
$$\mathcal{E}^1(x)=d(x,F_1)^2 \text{    and    }\mathcal{E}_1(x)=-d(x,F_1)^2.$$   Now, consider an interval $(a,b)$
contiguous to $F_1$ see Figure \ref{*fig1}.  We need to ensure that $f_1$ has derivative $0$ at $a$ and $b$.
  Note that $f_0\in (\mathcal{E}_1,\mathcal{E}^1)$ on $[a,b]$.
Now applying Lemma \ref{envelope1} with $f=f_0$, $\delta=2^{-3}$, $ \epsilon=1$ and $(f_l,f_u)=(\mathcal{E}_1,\mathcal{E}^1)$ we can define $f_1$ on $[a,b]$ so that
 \begin{equation}\label{f1 envelope}
f_1 \in (\mathcal{E}_1,\mathcal{E}^1) \text{    on    }[a,b]
\end{equation}
 \begin{equation}\label{loc monotone f1}
f_1 \text{    is locally monotonic on    } (a,b),
\end{equation}
and on any interval of monotonicity $[c,d]$ of $f_1$ we have
 \begin{equation}\label{f1 like phi}
f_1=K\pm(1-2^{-3})\phi.
\end{equation}

Note that by defining $f_1$ in this fashion on each contiguous interval of $F_1$ we ensure that $f_1$ is differentiable on $F_1$ with $f_1'=0$ on $F_1$.  It follows that (i) is satisfied for $n=1$ and it is easy to see that (iii) holds as well.  (\tr{Regarding these properties}, in the upcoming steps of the construction we will require that $f_n\in (\mathcal{E}_1,\mathcal{E}^1)$ on $[a,b]$ as well.
{This} will make sure that the limit function $f$ is differentiable on $F_1$ and its derivative vanishes there as desired.)
 \begin{figure}[!htb]
 \begin{center}
\includegraphics[trim={0 0 5cm 4cm}, width=0.95 \textwidth, clip]{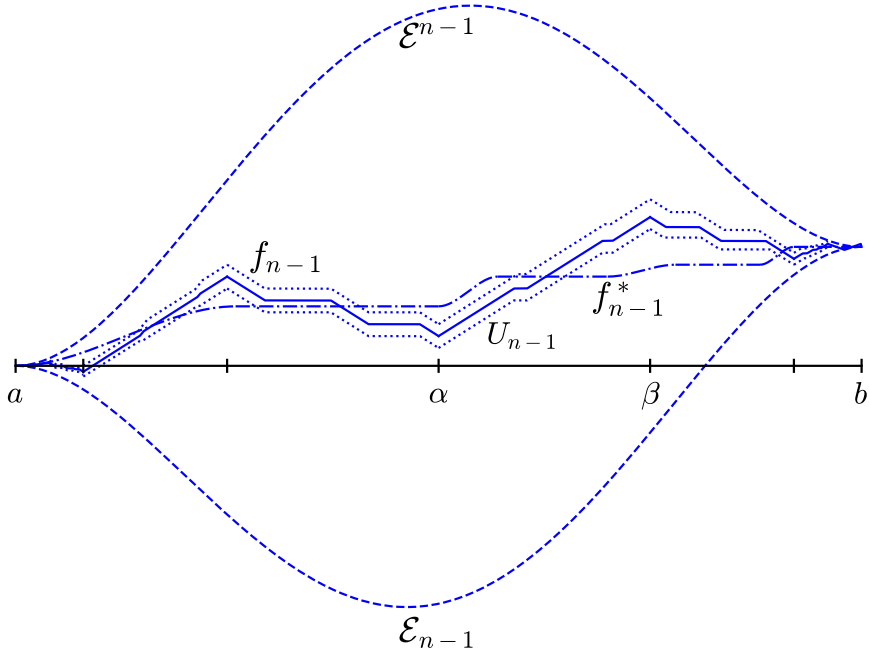}
\vspace*{-1cm}
\caption{$f_{n-1}$ on $(a,b)$}\label{*fig2}
\end{center}
\end{figure}

We next demonstrate that (iv) holds.  We assume without loss of generality that $\gamma_1 \ge \frac12$ and let $ x \in E^{\gamma_1,\delta_1} \cap (a,b)$.   
Choose a maximal interval of monotonicity $[c',d]$ of $f_1$ containing $x$.  We can assume without loss of generality that $x$ is in the left half of $[c',d]$ so that $ x \in [c',\frac{c'+d}2]$. Since $x\in (a,b)$, it is impossible that $x=a$.
In case of $c'\not=a$ we set $c=c'$. If $c'=a$ then we set $c=c'+\frac{x-c'}{2}=
a+\frac{x-a}{2}$.
In both cases  $ x \in [c,\frac{c+d}2]$. Using \eqref{loc monotone f1}
set $$e=\min\{  e'\in [a,c) : \text{
$f_{1}$ is monotone on $[e',c]$} \}.$$

Choose 
 \begin{equation}\label{*dddef}
\delta=\frac1{101}\min\{c-e,d-c, {\delta}_{1}\}.
\end{equation}
  Suppose first that $x \in [c+\delta,\frac{c+d}2]$.  In this case $f_1$ is monotone on $[x-\delta,x+\delta]$ and using the definition of $f_1$, the fact that $x \in E^{\gamma_1,\delta_1}$ and the fact that $\delta \le \delta_1$, we see that
 $|f_1(x)-f_1(y)|\ge (1-2^{-3})\gamma_1 |x-y|>(1-2^{-2})\gamma_1 |x-y|$ must hold for either $y=x-\delta$ or $y=x+\delta$.

 Now suppose that $x \in [c,c+\delta]$.  Since $x \in E^{\gamma_1,\delta_1}$ and $100\delta \le \delta_1$ we have that
 $\max\{|E\cap [x,x+100\delta]|,|E\cap [x-100\delta,x]|\}\ge100\gamma_1\delta$.  Suppose first that $|E\cap [x,x+100\delta]|\ge 100\gamma_1\delta$ and let $y=x+100\delta$.  In this case, since $[x,y]\subset [c,d]$, by the definition of $f_1$, we obtain $|f_1(x)-f_1(y)|\ge (1-2^{-3})\gamma_1 |x-y|>(1-2^{-2})\gamma_1 |x-y|$.  Now suppose that
 $|E\cap [x-100\delta,x]|\ge 100\gamma_1\delta$. Note that $[x-100\delta,c]\subset [e,c]$.  Setting $y=x-100\delta$, $S_1=\int_y^c(1-2^{-3})\boldmath{1}_E(t)\, dt$ and
 $ S_2=\int_{c}^x (1-2^{-3})\boldmath{1}_E(t)\,dt$, we get $|f_1(x)-f_1(y)|\ge S_1-S_2$. 
 On the other hand, we know that $S_2\le (1-2^{-3}) (x-c) \le (1-2^{-3})\delta$ and $S_1+S_2 =(1-2^{-3})|E\cap [y,x]|\ge  (1-2^{-3})100\gamma_1\delta$.  Using the fact that $\gamma_1\ge \frac12$, we see that $$|f_1(x)-f_1(y)|\ge S_1-S_2 \ge (1-2^{-3})(100\gamma_1-2)\delta > (1-2^{-2})100\gamma_1\delta=(1-2^{-2})\gamma_1|x-y|.$$

 Summing up, we see that in each of the two cases considered: $x \in [c,c+\delta]$ or $ x \in [c+\delta,\frac{c+d}2]$, we can choose
 $y=y_1(x)$ such that $\delta \le |x-y| \le \delta_1$ and $|f_1(x)-f_1(y)|\ge (1-2^{-2})\gamma_1|x-y|$.  Note that the definition of $\delta$ in \eqref{*dddef} ensures that $|x-y_1(x)|$ is bounded away from 0 on each compact subset of $(a,b)$.   This establishes (iv).

 Using the fact that $|x-y_1(x)|$ is bounded away from 0 on each compact subset of $(a,b)$, we see that we can define a continuous, non-negative function $r_1\le\mathcal{E}^1$ so that $r_1=0$ on $F_1$, $r_1>0$ on $G_1$ and
$r_1(x) ,r_1(y_1(x)) < 2^{-2} |x-y_1(x)|$ for all $x \in E^{\gamma_1,\delta_1}\cap G_1$ and $||r_1||_\infty \le 1/2$.  Letting $U_1$ be the vicinity of $f_1$ with radius $r_1$ we see that for any $g \in U_1$ we have $g \in (-\mathcal{E}_{1},\mathcal{E}_{1})$ on any interval $[a,b]$ contiguous to $F_1$.   It follows that (v)-(vii) have been established provided that we assume that at later steps $f_{m}\in U_{m} {\subset} U_{1}$ for $m>1$.

Now assume that we {have} already defined the functions $f_1,f_2,...,f_{n-1}$ and the decreasing sequence of vicinities $U_1,U_2,...,U_{n-1}$ with radii $r_1,r_2,\ldots, r_{n-1}$ for some $n\geq{2}$ so that they have the prescribed properties.
Since $r_{n-1}$ is continuous and positive on $G_{n-1}$, it follows that $r_{n-1}$ is bounded away from $0$ on all compact subsets of $G_{n-1}$.
Now we would like to define $f_n$ and $U_n$.
First we define an auxiliary function $f_n^*$.
Roughly {$f_n^*$} will be defined so that it has the same increment as $f_{n-1}$ in any interval of monotonicity of $f_{n-1}$, but has vanishing derivative on $F_{n}$.

To this end consider an interval $(a,b)$ contiguous to $F_{n-1}$. See Figure \ref{*fig2}.
On this figure the function $f_{n-1}$ is drawn with a continuous line,
the boundaries of the vicinity $U_{n-1}$ are marked with dotted lines, the envelope boundaries $\mathcal{E}_{n-1}$ and  $\mathcal{E}^{n-1}$ used in step $n-1$
are marked with dashed lines, finally the auxiliary function $f_{n-1}^{*}$
used at the previous step
 is marked with a dash-dot line.

By assumption we have
 \begin{equation}\label{fn-1ineq}
|f_{n-1}(x)-f_{n-1}(y)|\le (1-2^{-3(n-1)})|\phi(x)-\phi(y)|\text{    in    }  [a,b]
\end{equation}
and clearly $f_{n-1}\in (f_{n-1}-\frac{r_{n-1}}3,f_{n-1}+\frac{r_{n-1}}3)$ on $[a,b]$.  Let $ {\delta}'$ satisfy $2^{-3n} <  {\delta}' < 2^{-3(n-1)}$.  
Then by Lemma \ref{envelope2} used with $ {\varepsilon}=2^{-3(n-1)}$ and $ {\delta}= {\delta}'$ we can define $f_n^*$ on $[a,b]$ so that
 \begin{equation}
f_n^* \in (f_{n-1}-\frac{r_{n-1}}3,f_{n-1}+\frac{r_{n-1}}3) \text{    on    } [a,b]
\end{equation}
 \begin{equation}
f_n^*(a)=f_{n-1}(a) \text{    and    } f_n^*(b)=f_{n-1}(b)
\end{equation}
 \begin{equation}
({f_n^*})'=0 \text{    on    } F_n \cap (a,b)
\end{equation}
 \begin{equation}\label{fn*ineq}
|f_n^*(x)-f_n^*(y)|\le (1- {\delta}') |\phi(x)-\phi(y)| \text{    for all    } x,y\in [a,b].
\end{equation}

 \begin{figure}[!htb]
 \begin{center}
\includegraphics[trim={0 0 5cm 0.7cm}, width=0.95 \textwidth, clip]{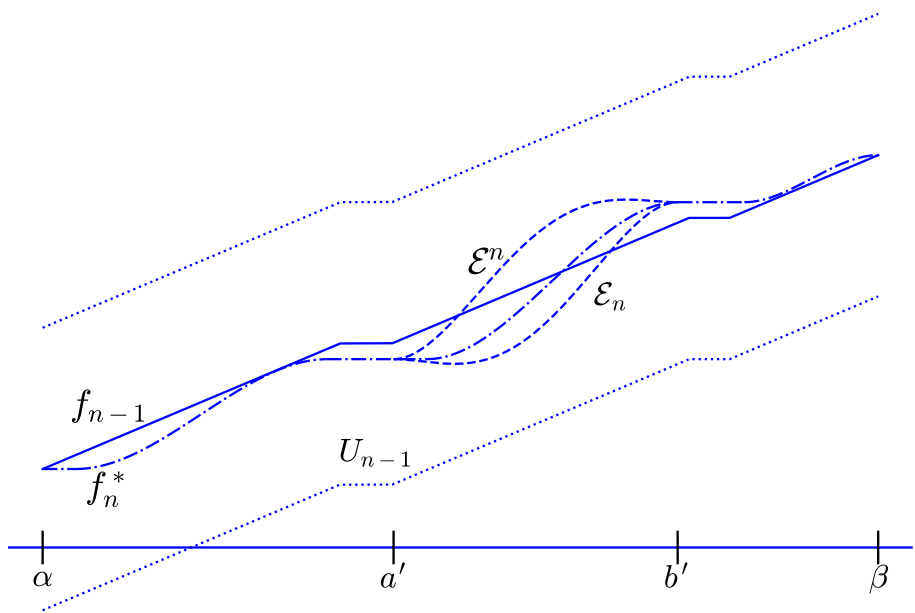}
\vspace*{-1cm}
\caption{$f_{n-1}$, $f_{n}^{*}$ on $( {\alpha}, {\beta})$}\label{*fig3}
\end{center}
\end{figure}

Now define $ \epsilon_n(x)=\min\{d(x,F_n)^2,\frac{r_{n-1}(x)}3\}$ and let $\mathcal{E}_n=f_n^*- \epsilon_n$ and $\mathcal{E}^n=f_n^*+ \epsilon_n$.
Let $(a',b')$ be contiguous to $F_n$ in $(a,b)$ so we have $f_n^* \in (\mathcal{E}_n,\mathcal{E}^n)$ on $[a',b']$. See Figure \ref{*fig3}.
Noting that (\ref{fn*ineq}) holds, we can apply Lemma \ref{envelope1} 
with $ {\varepsilon}= {\delta}'$ and $ {\delta}=2^{-3n}$ to define a function $f_n$ such that on each $[a',b']$ we have that $f_n \in (\mathcal{E}_n,\mathcal{E}^n)$, that $f_n$ is locally monotone and that on each interval of monotonicity we have $f_n=K\pm (1-2^{-3n})\phi$.

From our construction we see that (i)-(iii) hold.
On the other hand, (iv) is verified in a similar way as in the case $n=1$.
Finally, we consider (v)-(vii).    
We can define $r_n \in C[a,b]$ such that $r_n \le \min\{2^{-n}, \epsilon_n\}$ and $r_n >0$ on each interval $(a',b')$ contiguous to $F_{n}$.
Moreover, as $|x-y_n(x)|$ is bounded away from $0$ on each compact subset of each such contiguous interval $(a',b')$,  we can choose $r_n$ such that  if the vicinity $U_n$ has radius $r_n$ we have for any function $g\in U_n$ that $|g(x)-g(y_n(x))|>(1-2^{-n})\gamma_n|x-y_n(x)|$.
Note that $r_n \le  \epsilon_n$ guarantees $U_n\subseteq U_{n-1}$.
Thus,
\tr{for} a sufficiently small $r_n$ the condition $  r_n(x),r_n(y(x)) <   2^{-2n}\gamma_n |x-y_n(x)|$ is satisfied for all $x\in E^{\gamma_{n},\delta_{n}}${and therefore}
 (v) is verified. Moreover (vi) and (vii) follow easily as well.

By our earlier observations {this} concludes the proof: (v) guarantees that the sequence $(f_n)$ has a uniform limit function $f$, for which $ \Lip f(x)=1$ in $E$ by (iv) and (v).
 On the other hand, $ \Lip f(x)=0$ in the complement of $E$ by (ii) and (v), as $f$ has a vanishing derivative there by the choice of the vicinities.
\end{proof}

Note that sets of full measure are trivially UDT sets so an interesting consequence of Theorem \ref{*thUDTLip1} is that $G_\delta$ sets of full measure are $ \Lip 1$, yielding the following surprising corollary:

\tr{\begin{corollary} The set of irrational numbers is $\Lip 1$. That is, in terms of Dini derivatives, there exists a continuous function $f$ with $$\max\{D^{+}f(x), D^{-}f(x), -D_+f(x),-D_-f(x)\}=\mathbf{1}_{\mathbb{R}\setminus\mathbb{Q}}.$$ \end{corollary}}

\section{A weakly dense $G_\delta$ set which is not $ \Lip$ 1}\label{*secnotlip14}

Recall that in Section \ref{*secLip} we proved that $ \Lip 1$ sets are weakly dense, $G_\delta$ sets (Theorem \ref{Lip1weaklydense}).  In this section (Theorem \ref{notLip1weaklydense}) we show that weakly dense, $G_\delta$ sets need not be $ \Lip 1$.  For the proof of the theorem we will need the following:

 \begin{lemma}\label{E-ben}
Suppose that $E\subset {\ensuremath {\mathbb R}}$,
$f\colon {\ensuremath {\mathbb R}}\rightarrow {\ensuremath {\mathbb R}}$ and $ \Lip f= \mathbf{1}_E$.
 Then for every $x\in E$ and $ {\varepsilon}>0$ there is a $y\in E\cap (x- {\varepsilon},x+ {\varepsilon})$ for which $|f(x)-f(y)|>(1- {\varepsilon})|x-y|$.
\end{lemma}
 \begin{proof}
Take $y'\in  {\ensuremath {\mathbb R}}$ such that $|f(x)-f(y')|>\left(1-\frac{{\varepsilon}}{2}\right)|x-y'|$ and $|x-y'|< {\varepsilon}$.
We can assume that $ {\varepsilon}<1$ and $y'<x$.  
There is a $y\in E\cap (y',x)$ for which $|E\cap (y',y)| < \frac{{\varepsilon}}{2}|f(x)-f(y')|$ and Lemma \ref{novekedes} implies that 
 \begin{align*}
\frac{|f(x)-f(y)|}{|x-y|} &\ge \frac{|f(x)-f(y')|-|E\cap (y',y)|}{|x-y|} > \frac{|f(x)-f(y')|-\frac{{\varepsilon}}{2}|f(x)-f(y')|}{|x-y|} \\
&\ge \frac{\left(1-\frac{{\varepsilon}}{2}\right)^2|x-y'|}{|x-y'|} \ge 1- {\varepsilon}.
\end{align*}
\end{proof}

 \begin{remark}\label{*remwd}
Recall Definition \ref{weak dense}. It is easy to see that the following two statements are equivalent:
 \begin{itemize}
\item
 $E$ is weakly dense at $x$,
 \item
$\text{  for every $ {\varepsilon  }>0$ there is an $r\in (0, {\varepsilon})$ such that}$
 \begin{equation}\label{*eqwosd}
\max\Big \{\frac{|E\cap (x-r,x)|}r, \frac{|E\cap (x,x+r)|}r \Big \} > 1- {\varepsilon}.
\end{equation}
\end{itemize}
\end{remark}

 \begin{theorem}\label{notLip1weaklydense}
There exists a weakly dense, $G_\delta$  set $E\subset {\ensuremath {\mathbb R}}$ which is not $ \Lip 1$.
\end{theorem}
 \begin{proof}
We use recursion to define $E$.
Set $F_1:=[0,1]$.
Suppose that $n$ is a non-negative integer and
for some $(i_0,\ldots,i_n)\in \{1\}\times\ldots\times \{1,\ldots,4^n\}$
we have already defined a non-degenerate closed interval $F_{i_0,\ldots,i_n}$.
Let $U_{i_0,\ldots,i_n}$ be the left half of $F_{i_0,\ldots,i_n},$ that is
$$
U_{i_0,\ldots,i_n} := \left[\min F_{i_0,\ldots,i_n}, \frac{\min F_{i_0,\ldots,i_n} + \max{F_{i_0,\ldots,i_n}}}{2}\right].
$$
 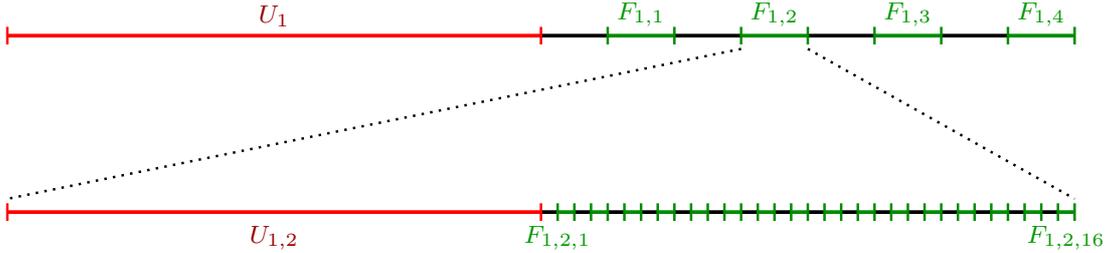
\begin{figure}[h]
\centering{
\resizebox{1.0\textwidth}{!}{


\scriptsize\begin{tikzpicture}[xscale=12]

\newcounter{magassag}
\setcounter{magassag}{2}

\makeatletter

\newcounter{szam}

\@whilenum\value{szam}<8\do{
\draw[-][draw=black, very thick] (.5+.5/8*\value{szam},\themagassag) -- (.5+0.5/8*\value{szam}+0.5/8,\themagassag);
\stepcounter{szam}
\stepcounter{szam}
\ }

\setcounter{szam}{0}
\newcounter{k}
\@whilenum\value{szam}<8\do{
\stepcounter{k}
\draw [draw=green, thick] (.5+.5/8*\value{szam}+0.5/8,\themagassag-0.1) -- (.5+.5/8*\value{szam}+0.5/8,\themagassag+0.1);
\draw (.5+.5/8*\value{szam}+0.5/8*3/2,\themagassag) node[above]{\textcolor[rgb]{0,0.6,0}{$F_{1,\thek}$}};
\stepcounter{szam}
\draw[-][draw=green, very thick] (.5+.5/8*\value{szam},\themagassag) -- (.5+0.5/8*\value{szam}+0.5/8,\themagassag);
\draw [draw=green, thick] (.5+.5/8*\value{szam}+0.5/8,\themagassag-0.1) -- (.5+.5/8*\value{szam}+0.5/8,\themagassag+0.1);
\stepcounter{szam}
\ }
\makeatother

\draw[-][draw=red, very thick] (0,\themagassag) -- (.5,\themagassag);
\draw [draw=red, thick] (0,\themagassag-.1) -- (0,\themagassag+0.1);
\draw [draw=red, thick] (0.5,\themagassag-.1) -- (0.5,\themagassag+0.1);
\draw (1/4,\themagassag) node[above]{\textcolor[rgb]{0.6,0,0}{$U_{1}$}};

\draw [draw=black, dotted, thick] (.5+.5/8*3,1.85) -- (0,0.15);
\draw [draw=black, dotted, thick] (.5+.5/8*4,1.85) -- (1,0.15);

\makeatletter
\setcounter{szam}{0}
\@whilenum\value{szam}<32\do{
\draw[-][draw=black, very thick] (.5+.5/32*\value{szam},0) -- (.5+0.5/32*\value{szam}+0.5/32,0);
\stepcounter{szam}
\stepcounter{szam}
\ }

\setcounter{szam}{0}
\@whilenum\value{szam}<32\do{
\draw [draw=green, thick] (.5+.5/32*\value{szam}+0.5/32,-0.1) -- (.5+.5/32*\value{szam}+0.5/32,0.1);
\stepcounter{szam}
\draw[-][draw=green, very thick] (.5+.5/32*\value{szam},0) -- (.5+0.5/32*\value{szam}+0.5/32,0);
\draw [draw=green, thick] (.5+.5/32*\value{szam}+0.5/32,-0.1) -- (.5+.5/32*\value{szam}+0.5/32,0.1);
\stepcounter{szam}
\ }
\makeatother

\draw (0.5+2/128,-0.05) node[below]{\textcolor[rgb]{0,0.6,0}{$F_{1,2,1}$}};
\draw (1-1/128,-0.05) node[below]{\textcolor[rgb]{0,0.6,0}{$F_{1,2,16}$}};

\draw[-][draw=red, very thick] (0,0) -- (.5,0);
\draw [draw=red, thick] (0,-.1) -- (0,0.1);
\draw [draw=red, thick] (0.5,-.1) -- (0.5,0.1);
\draw (1/4,0.-0.05) node[below]{\textcolor[rgb]{0.6,0,0}{$U_{1,2}$}};

\end{tikzpicture}

}}
\caption{The first two steps of the recursion} \label{fig_def}
\end{figure}
For every $i_{n+1}\in\{1,\ldots,4^{n+1}\}$ let 
 \begin{equation*}
 \begin{split}
F_{i_0,\ldots,i_n,i_{n+1}} := \left[\frac{(2\cdot 4^{n+1}-2i_{n+1}+1)\max U_{i_0,\ldots,i_n} + (2i_{n+1}-1)\max{F_{i_0,\ldots,i_n}}}{2\cdot4^{n+1}}, \right. \\
\left.\frac{(2\cdot4^{n+1}-2i_{n+1})\max U_{i_0,\ldots,i_n} + (2i_{n+1})\max{F_{i_0,\ldots,i_n}}}{2\cdot4^{n+1}} \right].
\end{split}
\end{equation*}
We define $U_{i_0,\dots,i_n}$ and $F_{i_0,\ldots,i_n}$ recursively in this way for every $n\in {\ensuremath {\mathbb N}}$ and $(i_0,\ldots,i_n)\in \{1\}\times\{1,\ldots,4\}\times\ldots\times \{1,\ldots,4^n\}$.
We are now ready to define $E$.  First define $$\mathcal{I}=\{1\}\times \{1,2,3,4\}\times ... \times \{1,2,...,4^n\}\times ...$$ and let $\mathcal{I}_1=\{(i_n)\in \mathcal{I} \,|\, i_n=1 \mbox{ for infinitely many }  n  \in \mathbb{N}\}$.  
Set
\begin{equation}\label{Fdef}
F := \bigcup_{(i_n)\in \mathcal{I}_1 }\bigcap_{n=1}^\infty F_{i_1,i_2,...,i_n}
\end{equation}
and
\begin{equation*}
\begin{split}
U := &\bigcup_{(i_n)\in \mathcal{I}} \bigcup_{n=0}^\infty U_{i_0,i_1,...,i_n}.
\end{split}
\end{equation*}

The set $F$ is a Cantor set minus countably many Cantor sets, hence it is $G_\delta$.
For every $n\in {\ensuremath {\mathbb N}}$ and $(i_0,\ldots,i_n)\in \{1\}\times\ldots\times\{1,\ldots,4^n\}$ there is an open set $U'_{i_0,\ldots,i_n}$ such that $U_{i_0,\ldots,i_n}\subset U'_{i_0,\ldots,i_n} \subset ( {\ensuremath {\mathbb R}}\setminus U)\cup U_{i_0,\ldots,i_n}$.
Thus $U$ is also $G_\delta$. This implies that
$$
E:=U\cup F
$$
is also $G_\delta$.

If $x\in U$ then $E$ is clearly weakly dense at $x$.
If $x\in F$ and $ {\varepsilon}>0$ then
using \eqref{Fdef}
take $n\in {\ensuremath {\mathbb N}}$ and $(i_0,\ldots,i_n)\in \{1\}\times\ldots\times\{1,\ldots,4^n\}$ such that $x\in F_{i_0,\ldots,i_n,1}$ and $ {\varepsilon}>\min\left\{|F_{i_0,\ldots,i_n}|,  4^{-n-1}\right\}$.
By the definition of $F_{i_0,\ldots,i_n,1}$ we have $ 4 \cdot4^{n+1} \left|F_{i_0,\ldots,i_n,1}\right| =\left|F_{i_0,\ldots,i_n}\right| $, hence
 \begin{align}
\frac{\left|\left(\min F_{i_0,\ldots,i_n}, x\right)\cap E\right|}{x-\min F_{i_0,\ldots,i_n}} &\ge \frac{\left|U_{i_0,\ldots,i_n}\right|}{\max F_{i_0,\ldots,i_n,1}-\min F_{i_0,\ldots,i_n}} = \frac{\frac{1}{2}\left|F_{i_0,\ldots,i_n}\right|}{\frac{1}{2}\left|F_{i_0,\ldots,i_n}\right|+2\left|F_{i_0,\ldots,i_n,1}\right|} \nonumber \\
&= \frac{\frac{1}{2}\left|F_{i_0,\ldots,i_n}\right|}{\left(\frac{1}{2}+\frac{2}{4\cdot4^{n+1}}\right) \left|F_{i_0,\ldots,i_n}\right|} = \frac{4^{n+1}}{4^{n+1}+1} \label{wd}\\
&= 1-\frac{1}{4^{n+1}+1} > 1- {\varepsilon}.\nonumber
\end{align}
By $\left(\min F_{i_0,\ldots,i_n}, x\right) \subset (x- {\varepsilon},x)$ and \eqref{wd} we obtain that $E$ is weakly dense at $x$.

\tr{We use proof by contradiction to show that $E$ is not $\Lip 1$.}  Assume the existence of a function $f\colon {\ensuremath {\mathbb R}}\rightarrow {\ensuremath {\mathbb R}}$ such that $ \Lip f = \mathbf{1}_E$.
We will show that there is a point $x^*\in  {\ensuremath {\mathbb R}}$ for which $0.1\le \Lip f(x^*)\le0.9$.
We will define $(i_0,i_1,\ldots)\in{\mathcal{I}}$ recursively such that $\{x^*\}=\bigcap_{n=0}^\infty F_{i_0,\ldots,i_n}$.
Set $a_0:=0$ and $i_0:=1$.
Suppose that $n\in {\ensuremath {\mathbb N}}$ and we have already defined a non-negative integer $a_{n-1}$ and $i_m\in \{1,\ldots,4^m\}$ for every $m\in\{0,\ldots,a_{n-1}\}$.
\tr{Let  $y_{n}=\min \left(F\cap F_{i_0,\ldots, i_{a_{n-1}}}\right)\in E$. Observe that   $\{y_{n}\}=\bigcap_{l=1}^{{\infty}}F_{i_{0},\ldots,i_{a_{n-1}},1^{l}}$, where $1^{l}=\ \ \!\!\!\!\!\underbrace{1,\ldots,1}_{l \text{ times  }}$.}
By Lemma \ref{E-ben} 
used with $x=y_{n}$ and $0< {\varepsilon}<\min \{|F_{i_0,\ldots, i_{a_{n-1}}}|,1/10 \}$
we can find an $x_n\in E$ satisfying $|y_{n}-x_{n}|< {\varepsilon}$ and
 \begin{equation}\label{*eqfxnyn}
|f(x_n)-f(y_n)|>0.9|x_n-y_n|.
\end{equation}
This implies that $x_{n}\in  F_{i_0,\ldots, i_{a_{n-1}}}.$
Since $x_{n}\not=y_{n}$ there exists $a_{n}>a_{n-1}$
 such that $x_{n}\in  F_{i_0,\ldots, i_{a_{n}-1}} {\setminus} F_{i_0,\ldots, i_{a_{n}-1},1}$
 while $y_{n}\in F_{i_0,\ldots, i_{a_{n}-1},1}.$
 The property $x_{n}\in F_{i_0,\ldots, i_{a_{n}-1}}$ defines $i_{m}$ for
 $m\in \{a_{n-1}+1,...,a_{n}-1\}.$ We might be able to find many $x_{n}$s satisfying the
 above property but we select an $x_{n}$ for which $a_{n}$ is minimal among the
 possible choices.
 Then $i_{m}=1$ for every  $m\in\{a_{n-1}+1,...,a_{n}-1\}.$

 If $k\in {\ensuremath {\mathbb N}}$, $(j_0,\ldots,j_k)\in \{1\}\times\ldots\times\{1,\ldots,4^k\}$, $j_{k+1},j'_{k+1}\in \{1,\ldots,4^{k+1}\}$, $j_{k+1} < j'_{k+1}$, $z\in F_{j_0,\ldots,j_k,j_{k+1}}$ and $z'\in F_{j_0,\ldots,j_k,j'_{k+1}}$, then by Lemma \ref{novekedes}  and the elementary fact
 \begin{equation}\label{alap}
0\le a < b \text{   and $0\le c$ implies $\frac{a   }{b} \le \frac{a+c}{b+c}$}
\end{equation}
we obtain
 \begin{equation}\label{jobbra nem}
 \begin{split}
\frac{|f(z)-f(z')|}{|z-z'|} &\le \frac{|E\cap [z,z']|}{|z-z'|} \\
&\le \frac{z-\min F_{j_0,\ldots,j_k,j_{k+1}} + |E\cap [z,z']| + \max F_{j_0,\ldots,j_k,j'_{k+1}}-z'} {\max F_{j_0,\ldots,j_k,j'_{k+1}} -\min F_{j_0,\ldots,j_k,j_{k+1}}}  \\
&\le \frac{(|j'_{k+1}-j_{k+1}|+1)\left|F_{j_0,\ldots,j_k,j_{k+1}}\right|}{(2|j'_{k+1}-j_{k+1}|+1)\left|F_{j_0,\ldots,j_k,j_{k+1}}\right|} \le \frac{2}{3}.
\end{split}
\end{equation}
This 
applied with  $z=y_{n}$, $z'=x_{n}$  and $k=a_{n}-1$
would imply that for $x_{n}\not\in U_{i_0,\ldots,i_{a_n-1}}$
we would have $|f(x_{n})-f(y_{n})|\leq \frac{2}{3}|x_{n}-y_{n}|$, contradicting
\eqref{*eqfxnyn}. Hence
 $x_n\in U_{i_0,\ldots,i_{a_n-1}}$.


For every $x\in U_{i_0,\ldots,i_{a_n-1}}$ and $y\in F_{i_0,\ldots,i_{a_n-1},4^{a_n}}$ again Lemma \ref{novekedes} and \eqref{alap} imply that
 \begin{align*}
\frac{|f(y)-f(x)|}{y-x} &\le \frac{|E\cap [x,y]|}{y-x} \\
&\le \frac{|E\cap [x,y]|+\left(x-\min F_{i_0,\ldots,i_{a_n-1}}\right)+\left(\max F_{i_0,\ldots,i_{a_n-1}}-y\right)}{y-x+\left(x-\min F_{i_0,\ldots,i_{a_n-1}}\right)+\left(\max F_{i_0,\ldots,i_{a_n-1}}-y\right)} \\
&\le \frac{\left|U_{i_0,\ldots,i_{a_n-1}}\right|+\sum \limits_{m=1}^{4^{{a_n}}} \left|F_{i_0,\ldots,i_{a_n-1},m}\right|}{\left|F_{i_0,\ldots,i_{a_n-1}}\right|} = \frac{3}{4}.
\end{align*}
Next we define $i_{a_{n}}$.
We select  an integer $i_{a_n}\in\{1,\ldots,4^{a_n}\}$ (let it be the least one)  such that for every $\widetilde{x}\in U_{i_0,\ldots,i_{a_n-1}}$ and $\widetilde{y}\in F_{i_0,\ldots,i_{a_n}}$ we have
 \begin{equation}\label{kicsi}
\frac{|f(\widetilde{y})-f(\widetilde{x})|}{\widetilde{y}-\widetilde{x}} \le 0.9.
\end{equation}

Since $x_n\in U_{i_0,\ldots,i_{a_n-1}}$ and 
$y_{n}\in E\cap F_{i_0,\ldots,i_{a_n-1},1}$
by \eqref{*eqfxnyn} we have that $i_{a_n}$ is larger than one.

As there are $w\in F_{i_0,\ldots,i_{a_n-1},i_{a_n}-1}$ and $v\in U_{i_0,\ldots,i_{a_n-1}}$ for which
 \begin{equation*}
\frac{|f(w)-f(v)|}{w-v} > 0.9 \text{   and hence   }
\end{equation*}
 \begin{equation}\label{*eqvw}
|[v,w]\setminus E| = w-v-|[v,w]\cap E| \le w-v-|f(w)-f(v)| \le \frac{w-v}{10},
\end{equation}
 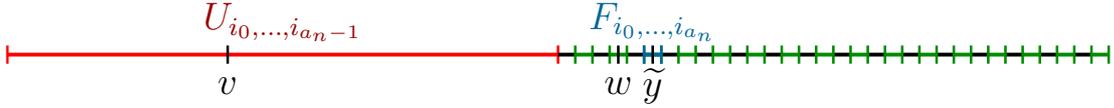
\begin{figure}[h]
\centering{
\resizebox{1.0\textwidth}{!}{\begin{tikzpicture}[xscale=12]

\newcounter{szam2}

\draw[-][draw=red, very thick] (0,0) -- (.5,0);
\draw [draw=red, thick] (0,-.1) -- (0,0.1);
\draw [draw=red, thick] (0.5,-.1) -- (0.5,0.1);
\draw (1/4,0) node[above]{\textcolor[rgb]{0.6,0,0}{$U_{i_0,\ldots,i_{a_n-1}}$}};

\makeatletter
\setcounter{szam2}{0}
\@whilenum\value{szam2}<32\do{
\draw[-][draw=black, very thick] (.5+.5/32*\value{szam2},0) -- (.5+0.5/32*\value{szam2}+0.5/32,0);
\stepcounter{szam2}
\stepcounter{szam2}
\ }

\setcounter{szam2}{0}
\@whilenum\value{szam2}<32\do{
\draw [draw=green, thick] (.5+.5/32*\value{szam2}+0.5/32,-0.1) -- (.5+.5/32*\value{szam2}+0.5/32,0.1);
\stepcounter{szam2}
\draw[-][draw=green, very thick] (.5+.5/32*\value{szam2},0) -- (.5+0.5/32*\value{szam2}+0.5/32,0);
\draw [draw=green, thick] (.5+.5/32*\value{szam2}+0.5/32,-0.1) -- (.5+.5/32*\value{szam2}+0.5/32,0.1);
\stepcounter{szam2}
\ }
\makeatother

\draw [draw=skyblue, very thick] (.5+10/128,0) -- (.5+12/128,0);
\draw [draw=skyblue, thick] (.5+10/128,-0.1) -- (.5+10/128,0.1);
\draw [draw=skyblue, thick] (.5+12/128,-0.1) -- (.5+12/128,0.1);

\draw [thick] (0.5+7/128,-0.1) node[below]{$w$} -- (0.5+7/128,0.1);
\draw [thick] (0.5+11/128,0.1) -- (0.5+11/128,-0.1);
\draw [thick] (0.5+11/128,0) node[above]{\color[rgb]{0,0.4,0.6}{$F_{i_0,\ldots,i_{a_n}}$}} -- (0.5+11/128,0) node[below]{$\widetilde{y}$};

\draw [thick] (0.2,-0.1) node[below]{$v$} -- (0.2,0.1);

\end{tikzpicture}}}
\caption{The position of $v$, $w$ and $\widetilde{y}$} \label{fig_vwy}
\end{figure}
for every $\widetilde{y}\in F_{i_0,\ldots,i_{a_n}}$ we obtain
 \begin{align}
\frac{|f(\widetilde{y})-f(v)|}{\widetilde{y}-v} &\ge \frac{|f(w)-f(v)|-|E\cap[w,\widetilde{y}]|}{\widetilde{y}-w+w-v} \nonumber \\
 &\text{   $\ge \dfrac{|f(w)-f(v)|-|E\cap[w,\widetilde{y   }]|}{\max F_{i_0,\ldots,i_{a_n}}-\min F_{i_0,\ldots,i_{a_n}-1}+w-v}$}  \nonumber \\
&\ge \frac{|f(w)-f(v)|-|E\cap[w,\widetilde{y}]|}{3\left|F_{i_0,\ldots,i_{a_n}}\right|+(w-v)} \nonumber \\
&\label{great}\ge \frac{|f(w)-f(v)|-2\left|F_{i_0,\ldots,i_{a_n}}\right|}{3\left|F_{i_0,\ldots,i_{a_n}}\right|+(w-v)} \ge \frac{0.9(w-v)-2\left|F_{i_0,\ldots,i_{a_n}}\right|}{3\left|F_{i_0,\ldots,i_{a_n}}\right|+(w-v)}  \\
&\ge \frac{0.9(w-v)-2\left|F_{i_0,\ldots,i_{a_n}}\right|}{4(w-v)}
\underset{\text{\eqref{*eqvw}}}{\ge} \frac{0.9\cdot 10|[v,w]\setminus E|-2\left|F_{i_0,\ldots,i_{a_n}}\right|}{4\cdot 10|[v,w]\setminus E|} \nonumber \\
&\text{   $\ge  \frac{0.9\cdot 10\left(\min F_{i_0,\ldots,i_{a_n-1   },1}-\max U_{i_0,\ldots,i_{a_n-1}}\right)-2\left|F_{i_0,\ldots,i_{a_n}}\right|}{4\cdot 10\left(\min F_{i_0,\ldots,i_{a_n-1},1}-\max U_{i_0,\ldots,i_{a_n-1}}\right)}$}   \nonumber \\
&= \frac{0.9\cdot 10\left|F_{i_0,\ldots,i_{a_n}}\right|-2\left|F_{i_0,\ldots,i_{a_n}}\right|}{4\cdot 10\left|F_{i_0,\ldots,i_{a_n}}\right|} = \frac{7}{40} > 0.1 \nonumber.
\end{align}
We define $a_n$ and $i_0,\ldots,i_{a_n}$ recursively for every $n\in {\ensuremath {\mathbb N}}$.

Set $\{ x^*\}:= \bigcap_{n=1}^\infty F_{i_0,\ldots,i_n}$.
From \eqref{great} we have $ \Lip f(x^*) > 0.1$.
We claim that
 \begin{equation}\label{small}
\frac{|f(\widehat{x})-f(x^*)|}{|\widehat{x}-x^*|} \le 0.9
\end{equation}
for every $\widehat{x}\in {\ensuremath {\mathbb R}}\setminus\{x^*\}$.
Suppose that an $\widehat{x}$ does not satisfy \eqref{small}. Since $f$ is continuous and it is constant on every complementary interval of the closure of $E$, we can assume that $\widehat{x}\in E$.
By \eqref{jobbra nem} there is a $k\in {\ensuremath {\mathbb N}}$ such that $\widehat{x}\in U_{i_0,\ldots,i_{k-1}}$ and $x^*\in F_{i_0,\ldots,i_{k-1},i_k}$.
Since $i_{a_n} > 1$ for  every  $n\in {\ensuremath {\mathbb N}}$
we have $x^*\neq \min \left(F\cap F_{i_0,\ldots,i_{k-1}}\right)=\bigcap_{l=1}^{{\infty}}F_{i_{0},\ldots,i_{k-1},1^{l}}$. This implies
 \begin{equation}\label{vege1}
 \begin{gathered}
\frac{|f(\widehat{x})-f(x^*)|}{|\widehat{x}-x^*|} \le \\
\le \max\left\{\frac{\left|f(\widehat{x})-f\left(\min \left(F\cap F_{i_0,\ldots,i_{k-1}}\right)\right)\right|}{\left|\widehat{x}-\min \left(F\cap F_{i_0,\ldots,i_{k-1}}\right)\right|}, \frac{\left|f(x^*)-f\left(\min \left(F\cap F_{i_0,\ldots,i_{k-1}}\right)\right)\right|}{\left|x^*-\min \left(F\cap F_{i_0,\ldots,i_{k-1}}\right)\right|}\right\}.
\end{gathered}
\end{equation}
Since \eqref{kicsi} shows that $k\neq a_n$ for any $n\in {\ensuremath {\mathbb N}}$, from the definition of $(a_n)_{n=0}^\infty$ we obtain
 \begin{equation}\label{vege2}
\frac{\left|f(\widehat{x})-f\left(\min \left(F\cap F_{i_0,\ldots,i_{k-1}}\right)\right)\right|}{\left|\widehat{x}-\min \left(F\cap F_{i_0,\ldots,i_{k-1}}\right)\right|} \le 0.9.
\end{equation}
Moreover \eqref{jobbra nem} implies that
 \begin{equation}\label{vege3}
\frac{\left|f(x^*)-f\left(\min \left(F\cap F_{i_0,\ldots,i_{k-1}}\right)\right)\right|}{\left|x^*-\min \left(F\cap F_{i_0,\ldots,i_{k-1}}\right)\right|} \le \frac{2}{3}.
\end{equation}
Hence by \eqref{vege1}, \eqref{vege2} and \eqref{vege3} we have
$$
\frac{|f(\widehat{x})-f(x^*)|}{|\widehat{x}-x^*|} \le 0.9,
$$
which is impossible.

Thus $ \Lip f(x^*)\neq \mathbf{1}_E(x^*)$, which is a contradiction.
\end{proof}

{\section{Open problems}}

\tr{As mentioned in the introduction, there are a number of problems in this area which are still open.  We list some of these below.}
\bigskip

\begin{itemize}

\item Characterize $\Lip 1$ sets.  This paper and \cite{[BHMVlipap]} provide progress in this direction, but there is still more work to be done.

\item Characterize the sets $E \subset \mathbb{R}$ for which there is a continuous function $f$ such that $\{x \,:\, \lip f(x) < \infty\}=E$.  See \cite{BHRZ} for a partial result on this problem.  The corresponding problem with $\Lip f$ in place of $\lip f$ turns out to be quite straightforward.

\item Characterize the sets $E$ which are sets of non-differentiability for continuous functions $f:\mathbb{R}\to \mathbb{R}$ such that $\lip f <\infty$ everywhere.  See \cite{Hanson2} for partial results in this direction.  

\end{itemize}

We are thankful for the referee's comments, which improved the presentation and readability of this paper.



\begin{thebibliography}{99}






  \bibitem{BaloghCsornyei}
{\sc Z.~M. Balogh and M.~Cs{\"o}rnyei}, {\em Scaled-oscillation and
  regularity}, Proc. Amer. Math. Soc., 134 (2006), pp.~2667--2675 (electronic).

 \bibitem{Banach}
{\sc S.~{Banach}}, {\em {\"Uber die Baire'sche Kategorie gewisser
  Funktionenmengen.}}, {Stud. Math.}, 3 (1931), pp.~174--179.
  
  \bibitem{BBT}
 {\sc A.~M.~Bruckner,  J.~B.~Bruckner and  B.~S.~Thomson,}  {\em Elementary Real Analysis} (2nd Edition),     ClassicalRealAnalysis.com (2008)  ISBN-13: 978-1434843678,  ISBN-10: 143484367X.
  
  \bibitem{BL}
  {\sc A.~Bruckner, J.~Leonard}, {\em Derivatives}, Amer. Math. Monthly, 73(4) Part II (1966), 24-56. 
  
   \bibitem{[BHMVlipap]}{\sc Z. Buczolich, B. Hanson, B. Maga and
G. V\'ertesy,} {\em Lipschitz one sets modulo sets of measure zero}, Math. Slovaca 70 (2020), No. 3, pp.~567–584.

\bibitem{BHMV2} {\sc Z. Buczolich, B. Hanson, B. Maga and 
G. V\'ertesy,} {\em Characterization of lip sets}, Journal of Mathematical Analysis and Applications 489 (2020), no. 2. 

\bibitem{bhmvdensity} {\sc Z. Buczolich, B. Hanson, B. Maga and G. V\'ertesy} {\em Strong one-sided density without uniform density}, in preparation. 

 \bibitem{BHRZ}
 {\sc Z.~Buczolich, B.~Hanson, M.~Rmoutil, and T.~Z\"urcher,} {\em On Sets where $ \lip f$ is finite},  Studia Math. 249 (2019), no. 1, 33–58.



 \bibitem{[Cheeger]}
{\sc J.~Cheeger}, {\em Differentiability of {L}ipschitz functions on metric
  measure spaces}, Geom. Funct. Anal., 9 (1999), pp.~428--517.



 \bibitem{Hanson}
{\sc B.~Hanson}, {\em Linear dilatation and differentiability of homeomorphisms
  of {$\mathbb{R}\sp n$}}, Proc. Amer. Math. Soc., 140 (2012), pp.~3541--3547.

 \bibitem{Hanson2}
{\sc B.~Hanson}, {\em Sets of Non-differentiability for Functions with Finite Lower Scaled Oscillation.} Real Analysis Exchange. {\bf 41(1)} (2016), pp.~87-100.


 \bibitem{[Keith]}
{\sc S.~Keith}, {\em A differentiable structure for metric measure spaces},
  Adv. Math., 183 (2004), pp.~271--315.
  
  \bibitem{M}
  {\sc A.~Maliszewski}, {\it Characteristic Functions and Products of Bounded Derivatives.} Proc. Amer. Math. Soc. 123, No. 7 (1995), 2203--2211. 
  
  \bibitem{MaZa} {\sc J. Malý and L. Zajíček,} {\it On Stepanov type differentiability theorems.} Acta Math. Hungar. 145 (2015), no. 1, 174--190.

 \bibitem{MaZi} {\sc J. Malý and O. Zindulka,} {\it Mapping analytic sets onto cubes by little Lipschitz functions.} Eur. J. Math. 5 (2019), no. 1, 91--105.
 
 \bibitem{PT}
 {\sc D.~Preiss, M.~Tartaglia,} {\it On Characterizing Derivatives.} Proc. Amer. Math. Soc. 123, No. 8 (1995), 2417--2420.
 
  \bibitem{[S]}
{\sc S. Saks},
Theory of the integral.
Second revised edition.
 Dover Publications, Inc., New York 1964 xv+343 pp.


 \bibitem{Stepanov}
{\sc W.~Stepanoff}, {\em \"{U}ber totale {D}ifferenzierbarkeit}, Math. Ann., 90
  (1923), pp.~318--320.
  
  \bibitem{W}
  {\sc C.~Weil,}, {\it On Properties of Derivatives.} Trans. Amer. Math. Soc. 114, No. 2 (Feb 1965), 363--376.
  
  \bibitem{Z} {\sc Z.~Zahorski,} {\it Sur la Premi\`{e}re D\'{e}riv\'{e}e.} Trans. Amer. Math. Soc. 69 (1950), 1--54. 



\end{thebibliography}
\end{document}